\documentclass[10pt, article]{amsart}
\usepackage{geometry} 
\usepackage{ae} % or {zefonts}
\usepackage[T1]{fontenc}
\usepackage[cp1250]{inputenc}
\usepackage{amsmath}
\usepackage{amssymb, amsfonts,amscd,verbatim}
\usepackage{MnSymbol}
\usepackage{mathrsfs} % for \mathscr
\usepackage{xypic}

\usepackage[normalem]{ulem}
\usepackage{hyperref}
\usepackage{indentfirst}
\usepackage{latexsym}
\input xy
\xyoption{all}

\usepackage{amsmath}    % need for subequations
\theoremstyle{plain}
\newtheorem{Pocz}{Poczatek}[section]
\newtheorem{Proposition}[Pocz]{Proposition}

\newtheorem{Theorem}[Pocz]{Theorem}
\newtheorem{Corollary}[Pocz]{Corollary}

\newtheorem{Lemma}[Pocz]{Lemma}
\newtheorem{Observation}[Pocz]{Observation}

\newtheorem{Example}[Pocz]{Example}
\newtheorem{Nonexample}[Pocz]{Nonexample}
\theoremstyle{definition}
\newtheorem{Definition}[Pocz]{Definition}

\theoremstyle{remark}
\newtheorem{Remark}[Pocz]{Remark}

\newcommand{\mesh}{\mathsf{mesh}}

\def\asdim{\mathrm{asdim}}

\DeclareMathOperator*{\st}{st}
\numberwithin{equation}{section}

\author{Logan ~ Higginbotham}
\address{University of Tennessee, Knoxville, USA}
\email{lhiggin3@vols.utk.edu}
\author{Thomas ~ Weighill}
\address{University of Tennessee, Knoxville, USA}
\email{tweighil@vols.utk.edu}

\title{Coarse quotients by group actions and the maximal Roe algebra}

\date{ \today
}


\begin{document}
%\fontsize{18}{20pt}\selectfont

\begin{abstract}
For a discrete metric space (or more generally a large scale space) $X$ and an action of a group $G$ on $X$ by coarse equivalences, we define a type of coarse quotient space $X_G$, which agrees up to coarse equivalence with the orbit space $X/G$ when $G$ is finite. We then restrict our attention to what we call coarsely discontinuous actions and show that for such actions the group $G$ can be recovered as an appropriately defined automorphism group $\mathsf{Aut}(X/X_G)$ when $X$ satisfies a large scale connectedness condition. We show that for a coarsely discontinuous action of a countable group $G$ on a discrete bounded geometry metric space $X$ there is a relation between the maximal Roe algebras of $X$ and $X_G$, namely that there is a $\ast$-isomorphism  
$C^\ast_{\max}(X_G)/\mathcal{K} \cong C^\ast_{\max}(X)/\mathcal{K} \rtimes G$,
where $\mathcal{K}$ is the ideal of compact operators. If $X$ has Property A and $G$ is amenable, then we show that $X_G$ has Property A, and thus the maximal Roe algebra and full crossed product can be replaced by the usual Roe algebra and reduced crossed product respectively in the above equation.
\end{abstract}

\maketitle

\section{Introduction}
An important notion in general and algebraic topology is that of a space equipped with a group action. For certain group actions (such as properly discontinuous actions) a lot can be said about the relationship between the space $X$ and the orbit space $X/G$. This paper makes some first steps towards a similar theory in the context of coarse geometry. 

Coarse geometry is the study of the large scale behaviour of metric spaces. The main motivating examples of metric spaces which exhibit interesting large scale behaviour are finitely generated groups with a word metric (this goes back to Gromov \cite{Gromov93}) and complete non-compact Riemannian manifolds (mainly with applications to index theory -- see \cite{RoeIndex}). The maps of interest in coarse geometry are the bornologous maps (also called large scale continuous maps). A (not necessarily continuous) map $f: X \rightarrow Y$ between metric spaces is called \textbf{large scale continuous} if is satisfies the following condition:
$$
\forall_{R > 0} \exists_{S > 0}\  d(x,y) \leq R \implies d(f(x), f(y)) \leq S.
$$
The main object of study in this paper is an action of a group $G$ on a discrete metric space $X$ such that the action of each $g \in G$ is a large scale continuous map. In fact, because of the existence of inverses in $G$, each action will be a so-called coarse equivalence.

Given such an action, one may want to consider the orbit space $X / G$ and put a metric on it which is natural from the point of view of large scale geometry. However, if $G$ is not finite, then (as mentioned above) the group $G$ itself has large scale behaviour which is not reflected in the space $X / G$. With this in mind, in this paper we introduce a different space (which we denote by $X_G$) which agrees with $X/G$ (with an appropriate metric) up to coarse equivalence only when $G$ is finite. After studying some general properties of the space $X_G$, we restrict our attention to a particular class of group actions, which we call coarsely discontinuous actions. These are the analogues of properly discontinuous actions for topological spaces. We show that when the action of $G$ is coarsely discontinuous and $X$ is an unbounded space which is coarsely one-ended (see Definition \ref{oneended}), then the group $G$ can be recovered from $X_G$ as an appropriately defined automorphism group $\mathsf{Aut}(X/X_G)$. 

An important $C^\ast$-algebra in the index theory of non-compact complete Riemannian manifolds is the Roe algebra. The Roe algebra is a coarse invariant (that is, invariant under coarse equivalences up to isomorphism) and is functorial with respect to proper large scale continuous maps at the level of $K$-theory (see for example \cite{RoeIndex}). Thus the Roe algebra is naturally an object of study in coarse geometry. The coarse Baum-Connes conjecture (see for example \cite{YuCoarseBC}) concerns an index map from the $K$-homology of Rips complexes on $X$ to the $K$-theory of the Roe algebra of $X$ (the conjecture is that this map is an isomorphism; it is false in general \cite{CounterEx}). A famous result of Yu \cite{YuEmbed} states that the conjecture is true for spaces which admit a coarse embedding into Hilbert space.

Gong, Wang and Yu introduced the related notion of maximal Roe algebra in \cite{GongMaximal} and formulated a version of the coarse Baum-Connes conjecture for this algebra in \cite{OyonoOyonoYu}.  
In this paper, we obtain some results relating the maximal Roe algebras of $X$ and $X_G$ for coarsely discontinuous actions. In particular, for such actions we obtain a short exact sequence
$$
0\to \mathcal{K} \to C^\ast_{\max}(X_G) \to (C^\ast_{\max}(X) / \mathcal{K}) \rtimes_{\alpha} G \to 0.
$$
where $\mathcal{K}$ is the algebra of compact operators. Note that the crossed product in the above sequence is the full crossed product. It is impossible in general to replace the full crossed product by the reduced crossed product and the maximal Roe algebra by the usual one in the sequence above (see Corollary \ref{YuCorollary}). However, we show that when $G$ is amenable and $X$ has Property A, then $X_G$ also has Property A and we have a short exact sequence
$$
0\to \mathcal{K} \to C^\ast(X_G) \to (C^\ast (X) / \mathcal{K}) \rtimes_{r, \alpha} G \to 0.
$$

Although almost all applications of coarse geometry are concerned with metric spaces, it is useful to consider a more general context -- that of large scale spaces in the sense of Dydak-Hoffland \cite{DH} (these spaces can also be viewed as abstract coarse spaces in the sense of Roe \cite{Roe}). This will become clear, for example, when we define the space $X_G$. 

In Section \ref{Sectypesofquotients}, we ask the following question: to what extent can the space $X_G$ can be considered a type of ``coarse quotient''? To answer this question, we first introduce and study what we call weak coarse quotient maps between large scale spaces. We show that the class of weak coarse quotient maps is closed under closeness and composition with coarse equivalences, and that the weak coarse quotient maps are precisely the maps which correspond to regular epimorphisms in the coarse category. Note that a notion of coarse quotient map has already been introduced by Zhang \cite{Zhang}, so we use the term ``weak'' here to avoid conflicting with that definition. When $G$ is finitely generated, the canonical map $X \to X_G$ will turn out to be a weak coarse quotient map. The final section is devoted to explicitly constructing metrics which induce the various large scale structures considered in the paper, including the one on $X_G$.

\section{Preliminaries}
\subsection{Large scale spaces}
The notion of large scale space (introduced in \cite{DH}) provides a general context for large scale geometry in the same way that uniform spaces provide a general context for questions of uniform continuity or convergence. A large scale space is a space equipped with a collection of families of subsets which are declared to be ``uniformly bounded''. To continue the analogy with uniform spaces above, the notion of large scale space is equivalent to the notion of coarse space in the sense of Roe \cite{Roe} in roughly the same way that the uniform covers definition of uniform space is equivalent to the entourage definition of uniform space (see \cite{DH}). Since the reader may not be familiar with the terminology of large scale structures we recall all the necessary definitions in this section, based mostly on \cite{DH}.  

Let $X$ be a set. Recall that the \textbf{star} $\st(B, \mathcal{U})$ of a subset $B$ of $X$ with respect to a family $\mathcal{U}$ of subsets of $X$ is the union of those elements of $\mathcal{U}$ that intersect $B$. More generally, for two families $\mathcal{B}$ and $\mathcal{U}$ of subsets of X, $\st(\mathcal{B}, \mathcal{U})$ is the family $\{\st(B, \mathcal{U}) \mid B \in \mathcal{B}\}$.

\begin{Definition}
A \textbf{large scale structure} $\mathcal{L}$ on a set $X$ is a nonempty
set of families $\mathcal{B}$ of subsets of $X$ (which we call the \textbf{uniformly bounded families} in $X$) satisfying the following conditions:
\begin{itemize}
\item[(1)] $\mathcal{B}_1 \in \mathcal{L}$ implies $\mathcal{B}_2 \in \mathcal{L}$ if each element of $\mathcal{B}_2$ consisting of more than one point is contained in some element of $\mathcal{B}_1$.
\item[(2)] $\mathcal{B}_1, \mathcal{B}_2 \in \mathcal{L}$ implies $\st(\mathcal{B}_1, \mathcal{B}_2) \in \mathcal{L}$.
\end{itemize}
\end{Definition}

Note that any uniformly bounded family can be extended to a uniformly bounded cover by adding all the singleton subsets to it, so we will often assume that a particular uniformly bounded family is in fact a cover. Also, note that if $\mathcal{U}$ and $\mathcal{V}$ are elements of a large scale structure $\mathcal{B}$, then so is $\mathcal{U} \cup \mathcal{V}$. By a \textbf{large scale space} (or \textbf{ls-space} for short), we mean a set equipped with a large scale structure. To be precise, we should write large scale spaces as pairs $(X, \mathcal{X})$, where $X$ is a set and $\mathcal{X}$ is a large scale structure. However, when we are dealing with only one large scale structure on each set, we will often write simply $X$ to mean the set equipped with its large scale structure. 

\begin{Example}
The canonical example of a large scale space is as follows. Let $(X, d)$ be an $\infty$-metric space. Define the uniformly bounded families in $X$ to be all those families $\mathcal{U}$ for which 
$$
\mathsf{mesh}(\mathcal{U}) = \mathsf{sup} \{ \mathsf{diam}(U) \mid U \in \mathcal{U} \} < \infty
$$
This is the large scale structure \textbf{induced} by the metric $d$. 
\end{Example}

A subset of a large scale space $X$ is called \textbf{bounded} if it is an element of some uniformly bounded family in $X$. A large scale space is called \textbf{coarsely connected} if every finite set is bounded (for example, every metric space is such). Every set admits a smallest coarsely connected large scale structure, namely those families of subsets which have only finitely many non-singleton sets.

\begin{Example}\label{exgroup}
Another important class of a large scale structures comes from group structures. If $G$ is a group, we can put a large scale structure on the underlying set of $G$ consisting of all refinements of covers of the form
$$
\{ g \cdot F \mid g \in G\}
$$
for some finite subset $F \subseteq G$. If $G$ is finitely generated, then this is the same large scale structure as the one induced by any word metric (see for example \cite{NowakYu}) on $G$. If $G$ is countable, this is the large scale structure induced by any discrete proper left-invariant metric on $G$ (see \cite{JSmith}).
\end{Example}

When dealing with quotients of metric spaces, it will often be easiest to first define (usually via a generating set) the large scale structure on the quotient, and then find a metric on the quotient space that induces this large scale structure. For this purpose, we need the following result, which was originally stated for coarse spaces by Roe in \cite{Roe} and later for large scale spaces in \cite{DH}. For completeness we include a proof. We say that a collection of families of subsets $\mathcal{A}$ \textbf{generates} a large scale structure $\mathcal{X}$ if $\mathcal{X}$ is the smallest large scale structure containing $\mathcal{A}$.

\begin{Theorem}[Theorem 1.8 in \cite{DH}] \label{metrizable}
Let $X$ be a large scale space. Then there exists an $\infty$-metric on $X$ which induces the large scale structure on $X$ if and only if the large scale structure is countably generated.
\end{Theorem}
\begin{proof}
Clearly if the large scale structure is induced by a metric $d$ then the countable set $(\mathcal{U}_i)_{i \in \mathbb{N}}$, where
$$
\mathcal{U}_i = \{ B(x, i) \mid x \in X \} ,
$$
generates the large scale structure. Suppose then that the large scale structure is generated by the countable set $(\mathcal{U}_i)_{i \in \mathbb{N}}$ of uniformly bounded families. We may assume that for every $i$, $\st(\mathcal{U}_i, \mathcal{U}_i)$ refines $\mathcal{U}_{i+1}$, and that $\mathcal{U}_0$ is actually a cover. For points $x$ and $y$ in $X$, define $d(x,y)$ to be the smallest $i$ for which there is an element of $\mathcal{U}_i$ containing both $x$ and $y$. One checks that this defines an $\infty$-metric that induces the large scale structure. 
\end{proof}

Finally, we establish some notation. For two families of subsets $\mathcal{U}$ and $\mathcal{V}$, we write $\mathcal{U} \leq \mathcal{V}$ if $\mathcal{U}$ refines $\mathcal{V}$, in which case we also say that $\mathcal{V}$ \textbf{coarsens} $\mathcal{U}$. For two points $x$ and $y$, we write $x\mathcal{U} y$, if $\mathcal{U}$ coarsens $\{ \{x, y\} \}$.

\subsection{Large scale continuous maps}
We now recall the right notion of ``continuous map'' between large scale spaces. Given a set map $f: X \rightarrow Y$ from an large scale space $X$ to an large scale space $Y$, we say that $f$ is \textbf{large scale continuous} or \textbf{ls-continuous} if for every uniformly bounded family $\mathcal{U}$ in $X$, the family
$$
f(\mathcal{U}) = \{f(U) \mid U \in \mathcal{U} \}
$$
is uniformly bounded in $Y$. These maps are the equivalent of bornologous maps for coarse spaces. In particular, a map $f: X \rightarrow Y$ from a metric space $X$ to a metric space $Y$ is large scale continuous if and only if for every $R> 0$ there exists an $S > 0$ such that the following holds for every $x_1, x_2 \in X$:
$$
d(x_1,x_2) \leq R \implies d(f(x_1), f(x_2)) \leq S.
$$

Let $f,g: X \rightarrow Y$ be two set maps between large scale spaces, not necessarily large scale continuous. We say that $f$ and $g$ are \textbf{close} and write $f \sim g$ if the family of subsets
$$
\{ \{f(x), g(x)\} \mid x \in X \}
$$
is uniformly bounded. Notice that any map that is close to a large scale continuous map is large scale continuous. Any uniformly bounded family which coarsens the family $\{ \{f(x), g(x)\} \mid x \in X \}$ is said to \textbf{witness} the closeness of $f$ and $g$.

If $A$ is a subset of a large scale space $X$, then there is a natural large scale structure which makes $A$ a \textbf{subspace} of $X$, namely the restriction of all uniformly bounded families in $X$ to $A$. Now let $f: X \rightarrow Y$ be a large scale continuous map. We say that $f$ is

\begin{itemize}
\item \textbf{coarsely surjective} if there is a uniformly bounded family $\mathcal{V}$ in $Y$ such that $Y \subseteq \st(f(X), \mathcal{V})$,
\item \textbf{a coarse embedding} if for every uniformly bounded family $\mathcal{V}$ in $Y$, the family
$$
f^{-1}(\mathcal{V})  = \{ f^{-1}(V) \mid V \in \mathcal{V} \}
$$
is uniformly bounded in $X$.
\item \textbf{a coarse equivalence} if it is both coarsely surjective and a coarse embedding.
\end{itemize}

Clearly the inclusion of a subspace into a large scale space is a coarse embedding. A large scale continuous map $f: X \rightarrow Y$ being coarsely surjective is clearly the same as requiring that the subspace inclusion $f(X) \rightarrow Y$ is a coarse equivalence. One can easily check that a large scale continuous map $f: X \rightarrow Y$ is a coarse equivalence if and only if there exists a large scale continuous map $g: Y \rightarrow X$ such that $fg$ and $gf$ are both close to the respective identity map. This suggests that coarse equivalences should be isomorphisms in an appropriate category. 

\subsection{The coarse category}
Various definitions exist in the literature for the coarse category, which was originally introduced by Roe in \cite{RoeIndex}. For example, Roe requires that all maps in this category be \textbf{proper} (the inverse image of a bounded set is bounded), while the authors of \cite{DikranjanZava} prefer to exclude this requirement since otherwise the category does not admit products. One common requirement however is that close maps are identified in the category. We will use the following definition for this paper.

\begin{Definition}
The category $\mathbf{Coarse}/ \sim$, called the \textbf{coarse category}, is the category whose objects are large scale spaces and whose morphisms are equivalence classes of large scale continuous maps under the closeness relation $\sim$.
\end{Definition} 

Composition in this category is defined in terms of representatives: $[\alpha] \circ [\beta] = [\alpha \circ \beta]$. One can check that this is well-defined. We recall the following lemma, which for coarse spaces is proved in \cite{DikranjanZava}. The proof of the first two facts is an easy adaptation of the proof presented there. The third fact follows from the remarks in the previous subsection.

\begin{Lemma} \label{epimono}
Let $f$ represent a morphism $[f]$ in $\mathbf{Coarse}/\sim$. Then
\begin{itemize}
\item $[f]$ is an epimorphism if and only if $f$ is coarsely surjective;
\item $[f]$ is a monomorphism if and only if $f$ is a coarse embedding;
\item $[f]$ is an isomorphism if and only if $f$ is a coarse equivalence.
\end{itemize} 
\end{Lemma} 

The coarse category (as defined in this paper) admits binary products. For two large scale spaces $X$ and $Y$, their product is the set $X \times Y$ equipped with the large scale structure consisting of all refinements of families of the form
$$
\mathcal{U} \times \mathcal{V} = \{ U\times V \mid U \in \mathcal{U},\ V \in \mathcal{V} \}
$$
for uniformly bounded families $\mathcal{U}$ in $X$ and $\mathcal{V}$ in $Y$.

\section{Group actions and $X_G$}
Given a set $X$ and a collection $\mathcal{X}$ of families of subsets  of $X$, we can consider the smallest large scale structure on $X$ containing $\mathcal{X}$, which we call the large scale structure \textbf{generated by $\mathcal{X}$}. Two possible constructions exist for this large scale structure, which we denote by $\overline{\mathcal{X}}$. One can either
\begin{itemize}
\item take the intersection of all large scale structures containing $\mathcal{X}$ (this is based on Proposition 2.12 in \cite{Roe}),
\item add the cover by singletons to $\mathcal{X}$ if necessary, close $\mathcal{X}$ under the star operation and then close the resulting collection under refinement.
\end{itemize}

\begin{Lemma} \label{generatingMap}
Let $f: X \rightarrow Y$ be a map between sets and let $\mathcal{X}$ be a collection of families of subsets of $X$. Then
$$
f(\overline{\mathcal{X}}) \subseteq \overline{f(\mathcal{X})}
$$
where $f(\mathcal{X}) = \{ f(\mathcal{U}) \mid \mathcal{U} \in \mathcal{X} \}$.
\end{Lemma}
\begin{proof}
Using the second construction above, this follows from the fact that for any two families $\mathcal{U}$ and $\mathcal{V}$ of subsets of $X$, 
$$
f(\st(\mathcal{U}, \mathcal{V})) \leq \st(f(\mathcal{U}), f(\mathcal{V})).
$$
\end{proof}

Let $X$ be a large scale space. By an \textbf{action of a group $G$ on $X$ by coarse equivalences} we mean an action of $G$ on the underlying set of $X$ such that every $g \in G$ acts as a large scale continuous map. 

\begin{Definition}\label{defXG}
Let $X$ be a large scale space with large scale structure $\mathcal{X}$ and let $G$ be a group acting on $X$ by coarse equivalences. Let $\mathcal{X}_G$ be the large scale structure on $X$ generated by $\mathcal{X}$ together with all families of the form
$$\{ \{x, gx\} \mid x \in X \} $$
for some $g \in G$. We denote the set $X$ together with the large scale structure $\mathcal{X}_G$ by $X_G$, and we denote by $p_G: X \rightarrow X_G$ the identity set map.
\end{Definition}

One way to view $X_G$ is that it is the underlying set of $X$ equipped with the smallest large scale structure which makes the action of each $g\in G$ close to the identity. Since the term coarse quotient already exists in the literature \cite{Zhang}, and the map $p_G: X \rightarrow X_G$ is not in general a coarse quotient mapping under the definition found in \cite{Zhang}, we will rarely refer to $X_G$ as ``the coarse quotient of $X$ by $G$'' in this paper, instead opting to use the notation $X_G$ to make the distinction clear.

\begin{Remark}\label{XGrem1}
It is easy to check that for any finite subset $F \subseteq G$, the family
$$
\{F\cdot x \mid x \in X \} = \{ \{ f \cdot x \mid f \in F \} \mid x \in X \}
$$
is uniformly bounded in $X_G$. In particular, let $G$ be a group and $|G|$ the underlying set of $G$ with either the large scale structure consisting only of families of singletons or the smallest coarsely connected large scale structure. Let $G$ act on $|G|$ by right translation, that is, $g \cdot h = hg^{-1}$. Then $|G|_G$ is $G$ with the large scale structure coming from the group structure (see Example \ref{exgroup}).
\end{Remark}

\begin{Lemma}\label{XGUF}
Let $(X, \mathcal{X})$ be a large scale space and let $G$ be a group acting on $X$ by coarse equivalences. Then the large scale structure on $X_G$ is precisely the collection $\mathcal{X}'$ of refinements of families of the form $\st(\mathcal{U}, \mathcal{F})$, where $\mathcal{U}$ is a uniformly bounded family in $X$ and $\mathcal{F}$ is of the form
$$
 \{ F \cdot x \mid x \in X \}.
$$
for some finite subset $F \subseteq G$.
\end{Lemma}
\begin{proof}
Since $\mathcal{X}'$ contains $\mathcal{X}$ and all families of the form $\{ \{x, gx\} \mid x \in X \}$, it is enough to show that $\mathcal{X}'$ is a large scale structure, that is, closed under stars. Since
$$
\st(\st(\mathcal{U}_1, \mathcal{F}_1), \st(\mathcal{U}_2, \mathcal{F}_2)) \leq 
\st(\st(\st(\st(\mathcal{U}_1, \mathcal{F}_1), \mathcal{F}_2), \mathcal{U}_2), \mathcal{F}_2)
$$
it is enough to prove that both $\st(\st(\mathcal{U}_1, \mathcal{F}_1), \mathcal{F}_2)$ and $\st(\st(\mathcal{U}_1, \mathcal{F}_1), \mathcal{U}_2)$ are in $\mathcal{X}'$ for any $\mathcal{U}_1, \mathcal{U}_2 \in \mathcal{X}$ and any two families
$$
\mathcal{F}_1 = \{ F_1 \cdot x \mid x \in X \},\ \mathcal{F}_2 = \{ F_2 \cdot x \mid x \in X \}.
$$
where $F_1$ and $F_2$ are finite subsets of $G$. For the first, we have
$$
\st(\st(\mathcal{U}_1, \mathcal{F}_1), \mathcal{F}_2) \leq 
\st(\mathcal{U}_1, \mathcal{F}_3)
$$
where 
$$
\mathcal{F}_3 = \{ F_1 \cdot F_1^{-1} \cdot F_2 \cdot x \mid x \in X \}.
$$
For the second one, let $\mathcal{V}$ be the element of $\mathcal{X}$ (using the fact that $F_1$ is finite)
$$
\{ g \cdot U \mid U \in \mathcal{U}_2 \wedge g \in F_1 \cdot F_1^{-1} \}.
$$
Then
$$
\st(\st(\mathcal{U}_1, \mathcal{F}_1), \mathcal{U}_2) \leq 
\st(\st(\mathcal{U}_1, \mathcal{V}), \mathcal{F}_1)
$$
which gives the required result.
\end{proof}

Some readers may prefer to use Lemma \ref{XGUF} as the definition of the large scale structure on $X_G$. Definition \ref{defXG} emphasises the universal property of $X_G$, whereas in practice the one provided by the lemma is most useful. We now briefly consider the case when $G$ is finite. In particular, we show that $X_G$ is the same, up to coarse equivalence, as $X/G$ with an appropriate large scale structure.

\begin{Lemma}\label{Gfinitestrong}
Let $(X, \mathcal{X})$ be a large scale space and let $G$ be a finite group acting on $X$ by coarse equivalences. Let $q: X \rightarrow X/G$ be the quotient map onto the orbit space, and let $q(\mathcal{X})$ be the collection of all images of elements of $\mathcal{X}$ under $q$. Then $q(\mathcal{X})$ is a large scale structure.
\end{Lemma}
\begin{proof}
Let $\mathcal{U}, \mathcal{V} \in \mathcal{X}$. Let $\mathcal{U}' = \bigcup_{g \in G} g(\mathcal{U})$ and $\mathcal{V}' = \bigcup_{g \in G} g(\mathcal{V})$ respectively, and note that $\mathcal{U}'$ and $\mathcal{V}'$ are each in $\mathcal{X}$ because $G$ is finite. Then $\mathcal{U}$ and $\mathcal{V}$ refine $\mathcal{U}'$ and $\mathcal{V}'$ respectively, and one can check that
$$
\st(q(\mathcal{U}'), q(\mathcal{V}'))  = q(\st(\mathcal{U}', \mathcal{V}')) \in q(\mathcal{X}).
$$
\end{proof}

\begin{Proposition}\label{finiteG}
Let $(X, \mathcal{X})$ be a large scale space and let $G$ be a finite group acting on $X$ by coarse equivalences. Let $X/G$ be the orbit space with the large scale structure $q(\mathcal{X})$, where $q$ is the quotient map $q: X \rightarrow X/G$. Then the natural map $p: X_G \rightarrow X/G$ is a coarse equivalence.
\end{Proposition}
\begin{proof}
Since the image of every family $ \{ \{x, gx\} \mid x \in X \} $ under $q$ is a family of singletons, the map $p$ is a large scale continuous map by Lemma \ref{generatingMap}, which is moreover surjective. It remains to show it is a coarse embedding. By Proposition \ref{Gfinitestrong}, every uniformly bounded family $\mathcal{U}$ in $X/G$ is the image under $q$ of a uniformly bounded family $\mathcal{U}'$ in $X$. Then 
$$
q^{-1}(\mathcal{U}) \leq \st(\mathcal{U}', \mathcal{G})
$$
where $\mathcal{G} = \{ \{x, gx\} \mid g \in G,\ x \in X \}$. The family $\st(\mathcal{U}', \mathcal{G})$ is uniformly bounded in $X_G$ because $G$ is finite, which completes the proof.
\end{proof}

\begin{Proposition}\label{obsfinitemetric}
Let $X$ be a metric space and $G$ a finite group acting on $X$ by coarse equivalences. Then the large scale structure $q(\mathcal{X})$ on the orbit space $X/G$, where $q$ is the quotient map $q: X \rightarrow X/G$ and $\mathcal{X}$ is the large scale structure on $X$, is induced by the Hausdorff metric on orbits, that is, 
$$d_{\mathsf{Haus}}([x],[y]) = \min \{ \sup_{x' \in [x]} \inf_{y' \in [y]}  d(x', y'), \sup_{y' \in [y]} \inf_{x' \in [x]} d(x', y') \}.$$
Suppose further than $G$ acts by isometries. Then $q(\mathcal{X})$ is also induced by
$$d_{\mathsf{min}}([x],[y]) = \min \{ d_X(x',y') \mid x' \in [x], y' \in [y] \} = \min \{d_X(x, g\cdot y) \mid g \in G \}.$$
\end{Proposition}
\begin{proof}
If $d_{\mathsf{Haus}}([x],[y]) \leq R$, then there are $x' \in [x], y' \in [y]$ such that $d_X(x',y') \leq R$, so clearly every cover which is uniformly bounded with respect to $d_{\mathsf{Haus}}$ is contained in $q(\mathcal{X})$. On the other hand, if $d_X(x,y) \leq R$ for $x,y \in X$, then there is an $S > 0$ depending only on $R$ such that $d_X(g \cdot x, g\cdot y) \leq S$ for all $g \in G$. It follows that $d_{\mathsf{Haus}}([x],[y]) \leq S$, so every element of $q(\mathcal{X})$ is uniformly bounded with respect to $d_{\mathsf{Haus}}$. If $G$ acts by isometries, then $d_{\mathsf{min}}([x],[y])$ is indeed a metric, and it is easy to see that it induces $q(\mathcal{X})$. 
\end{proof}

We recall the definition of coarsely light map from \cite{WeighillDydakMonotone}. For $A$ a subset of a large scale space $X$ and $\mathcal{U}$ a cover of $X$, let $A_{\mathcal{U}}$ be the set of equivalence classes of $A$ under the equivalence relation: $x \sim x'$ if and only if there exists a finite sequence $x_0, \ldots, x_n$ of elements of $A$ with $x \mathcal{U} x_0$, $x_n \mathcal{U} x'$ and $x_i \mathcal{U} x_{i+1}$ for all $0 \leq i < n$. A large scale continuous map $f: X \rightarrow Y$ is called \textbf{coarsely light} if for every pair of uniformly bounded covers $\mathcal{V}$ of $Y$ and $\mathcal{U}$ of $X$, the family of subsets
$$
\bigcup_{V \in \mathcal{V}} f^{-1}(V)_\mathcal{U}
$$
is uniformly bounded in $X$. Equivalently, $f$ is coarsely light if for every uniformly bounded family $\mathcal{V}$ in $Y$, the family of subsets $f^{-1}(\mathcal{V})$ has asymptotic dimension zero uniformly (see \cite{WeighillDydakMonotone} for details).

\begin{Proposition}\label{XGlight}
Let $(X, \mathcal{X})$ be a large scale space and let $G$ be a group acting on $X$ by coarse equivalences. Then the identity set map $p_G: X \rightarrow X_G$ is coarsely light.
\end{Proposition}
\begin{proof}
Let $\st(\mathcal{V}, \mathcal{F})$ be a uniformly bounded cover of $X_G$, where $\mathcal{V}$ is a uniformly bounded cover of $X$ and $\mathcal{F} = \{ F \cdot x \mid x \in X \}$ for a finite subset $F \subseteq G$ containing the identity (see Lemma \ref{XGUF}). Since for any $\st(V, \mathcal{F}) \in \st(\mathcal{V}, \mathcal{F})$, 
$$
\st(V, \mathcal{F}) \subseteq \bigcup\limits_{g \in F \cdot F^{-1}} g \cdot V,
$$
it follows that every element of $\st(\mathcal{V}, \mathcal{F})$ is contained in a union of $n = |F \cdot F^{-1}|$ elements of the family $\mathcal{V}' = \bigcup_{g \in F \cdot F^{-1}} g \cdot \mathcal{V}$, which is uniformly bounded in $X$. Let $\mathcal{U}$ be a uniformly bounded cover in $X$, which we may assume coarsens $\mathcal{V}'$. Then for any $W \in \st(\mathcal{V}, \mathcal{F})$, each element of $W_\mathcal{U}$ is contained in an element of the uniformly bounded cover
$$
\underbrace{\st(\st( \cdots \st(}_{n-1 \text{ times}} \mathcal{U}, \mathcal{U}), \mathcal{U}) \cdots, \mathcal{U}), \mathcal{U}).
$$
\end{proof}

\begin{Corollary}
Let $(X, \mathcal{X})$ be a large scale space and let $G$ be a group acting on $X$ by coarse equivalences. Then $\asdim\ X \leq \asdim\ X_G$. 
\end{Corollary}
\begin{proof}
This follows from Proposition 9.1 in \cite{WeighillDydakMonotone}. 
\end{proof}

In the case when $G$ is finite, we can apply a result of Kasprowski \cite{Kasprowski} to get a better result. Kasprowski proves that for a proper metric space $X$  and a finite group $G$ acting on $X$ by isometries, $X/G$ with the metric 
\begin{equation}\label{minmetric}
d([x], [x']) = \min_{g \in G} d(x, gx')
\end{equation}
has asymptotic dimension equal to that of $X$. We get the following corollary for $X_G$.

\begin{Corollary}[of Theorem 1.1 in \cite{Kasprowski}]
Let $X$ be a proper metric space and let $G$ be a finite group acting on $X$ by coarse equivalences. Then $\asdim X_G = \asdim X$. 
\end{Corollary}
\begin{proof}
As was already observed in \cite{DydakVirk}, the metric
$$
d'_X(x, x') = \sum_{g \in G} d(g\cdot x, g \cdot x') 
$$
on $X$ induces the same large scale structure as the original metric on $X$, and the group $G$ acts on $X$ by isometries with respect to the metric $d'$. Thus we may reduce to the case when $G$ acts by isometries (since asymptotic dimension is invariant under coarse equivalence). Let $d$ be the metric  on $X/G$ defined in (\ref{minmetric}), which by Proposition \ref{obsfinitemetric} induces the large scale structure $q(\mathcal{X})$ on $X/G$, where $q: X \rightarrow X/G$ is the quotient map. Applying Kaprowski's result, we have that $\asdim\ X/G = \asdim\ X$, and since asymptotic dimension is invariant under coarse equivalence, we obtain the result by Proposition \ref{finiteG}. 
\end{proof}

When $X$ is a metric space and $G$ is a countable group then the large scale structure on $X_G$ is countably generated, hence metrizable. In general, an explicit construction of a metric inducing the large scale structure on $X_G$ is cumbersome (see Section \ref{secmetrizable}), but in the case when $G$ acts by isometries, a construction is slightly easier to write down. 

\begin{Proposition}
If $X$ is a metric space and $G$ is a finitely generated group which acts on $X$ by isometries, then the large scale structure on $X_G$ is induced by the metric
$$
d(x, x') = \inf_{g \in G} \{ d_X(x, gx') + |g| \}
$$ 
where $|g|$ is the word-length of $g$ and $d_X$ is the metric on $X$. 
\end{Proposition}
\begin{proof}
Because $G$ acts by isometries on $X$, $d$ is indeed a metric. If $\mathcal{U}$ is a uniformly bounded family with respect to $d$, with $\mathsf{mesh}(\mathcal{U}) < R$, then let $F$ be the finite set $\{ g \in G \mid |g| \leq R\}$ and let $\mathcal{B}$ be the cover of $X$ by $R$-balls in the original metric $d_X$. Then since $d(x, x') \leq R \implies \exists_{g\in F} d_X(x, gx') \leq R$,  it is easy to check that $\mathcal{U}$ refines the cover $\st(\mathcal{B}, \mathcal{F})$, where $\mathcal{F} = \{F \cdot x \mid x \in X\}$. On the other hand, any cover of the form  $\st(\mathcal{U}, \mathcal{F})$ where $\mathcal{F} = \{F \cdot x \mid x \in X\}$ for a finite subset $F \subseteq G$, refines the set of $R$-balls with respect to $d$, where
$$
R = \max \{|g| \mid g\in F \cdot F^{-1} \} + \mathsf{mesh}(\mathcal{U}).
$$
\end{proof}

\section{Coarsely discontinuous actions}
We now restrict our attention to the analogue of properly discontinuous actions on topological spaces, which we call coarsely discontinuous actions.

\begin{Definition}
Let $X$ be a large scale space and let $G$ be a group acting on $X$ by coarse equivalences. We say that the action of $G$ is \textbf{coarsely discontinuous} if for every uniformly bounded family $\mathcal{U}$ and every element $g \in G \setminus \{e\}$, there is a bounded set $K$ such that for every $U \in \mathcal{U}$ with $U \cap K = \varnothing$, we have $
U \cap g \cdot U = \varnothing.$
\end{Definition}

If $X$ is a metric space, then to say that an action of $G$ on $X$ is coarsely discontinuous is clearly the same as to say that for each $g \in G \setminus \{e\}$ and each $R>0$ there is a bounded set $K$ such that $d(x, g\cdot x) \geq R$ for all $x \notin K$. Or, more succintly, $d(x, g\cdot x) \to \infty$ as $x \to \infty$
for every $g \neq e$.

\begin{Example}
Let $X$ be a compact metric space. By the \textbf{cone} $CX$ on $X$ we mean the quotient $X \times [0, \infty) / \sim$, where $\sim$ is the equivalence relation generated by the pairs $\{ ((x, 0), (x',0)) \mid  x,x' \in X\}$. We can turn $CX$ into a metric space by choosing a continuous weight function $\Phi: [0, \infty) \to [0, \infty)$ with $\Phi(t) = 0 \Leftrightarrow t = 0$ and defining 
$$
d([(x,t)], [(x',t')]) = \inf \left\lbrace \sum\limits_{j=0}^{n-1} |t_j - t_{j+1}| + \max \{\Phi(t_j), \Phi(t_{j+1}) \} d_X(x_j, x_{j+1})  \right\rbrace
$$
where the infimum is taken over all finite sequences $(x_j, t_j)_{0 \leq j \leq n}$ of points in $CX$ with $(x_0, t_0) \sim (x,t)$ and $(x_n, t_n) \sim (x', t')$. One checks that this is a metric. If $X$ is a path metric space, then this is the same metric defined in (3.46) of \cite{RoeComplete}.

If $G$ is a group which acts on $X$ properly discontinuously by isometries, then there is a natural action of $G$ on $CX$ given by $g \cdot (x,t) = (g \cdot x, t)$. It is easy to check this action is by isometries. If $\Phi(t) \to \infty$ as $t \to \infty$, then the action is also coarsely discontinuous. Indeed, let $R > 0$ and $g \in G \setminus \{e\}$. Define
$$
k_g = \min \{ d_X(x, g\cdot x) \mid  x \in X\} > 0.
$$
Pick $t_0 > R$ such that $\Phi(t) > R/k_g$ for all $t > t_0 - R$. If $(x,t) \in CX$ with $t > t_0$, then for any sequence $(x_j, t_j)_{0 \leq j \leq n}$ with $(x,t) = (x_0, t_0)$ and $(g \cdot x, t) = (x_n, t_n)$, if $t_j \geq t_0$ for all $j$ then
$$
\sum\limits_{j=0}^{n-1} |t_j - t_{j+1}| + \max \{\Phi(t_j), \Phi(t_{j+1}) \} d_X(x_j, x_{j+1})  \geq \sum\limits_{j=0}^{n-1} R/k_g \cdot  d_X(x_j, x_{j+1}) \geq R
$$
otherwise $t_m < t_0$ for some $m$ and thus
$$
\sum\limits_{j=0}^{n-1} |t_j - t_{j+1}| + \max \{\Phi(t_j), \Phi(t_{j+1}) \} d_X(x_j, x_{j+1}) \geq \sum\limits_{j=0}^{m-1} |t_j - t_{j+1}| \geq R
$$
by the triangle inequality. Thus outside of the bounded set $X \times [0, t_0]/\sim$, $d((x,t), g\cdot (x,t)) \geq R$ as required.
\end{Example}

\begin{Example}\label{Gleft}
Let $G$ be a group equipped with a discrete proper left-invariant metric (such as a word-length metric). Denote $d(e,g)$ by $|g|$ so that $d(g,h) = |g^{-1}h|$. Consider the action of $G$ on itself as a metric space by left translation: $g\cdot x = gx$. This action is an action by isometries, hence in particular by coarse equivalences. Since $d(x, gx) = |x^{-1}gx|$, this action will be coarsely discontinuous if for every $R > 0$ and $g \neq e$,  $\{ x \mid |x^{-1}gx| \leq R \}$ is finite. The ball of radius $R$ around the identity element is itself finite, so this is clearly equivalent to requiring that for each $x \in G$, $\{ y \in G \mid y^{-1}gy = x^{-1}gx \}$ is finite. Since $x^{-1}gx = y^{-1}gy$ if and only if $xy^{-1}$ is in the centralizer of $g$, we have that the action of $G$ on itself by left translation is coarsely  discontinuous if the centralizer of every $g \in G \setminus \{ e\}$ is finite. Conversly, if the centralizer of some $g \in G \setminus \{ e \}$ is infinite, then $d(x, gx) = |g|$ for infinitely many $x \in G$, so the action can't be coarsely discontinuous. Thus the action of $G$ on itself by left translation is coarsely discontinuous if and only if the centralizer of every non-identity element of $G$ is finite.
\end{Example}

\begin{Nonexample}\label{Gright}
Again consider a finitely generated group equipped with discrete proper left-invariant metric, but this time consider the action by right translation: $g\cdot x = xg^{-1}$. The action of each $g \in G$ by right translation is close to the identity since $d(x, x g^{-1}) \leq |g^{-1}|$, so it follows that this action is by coarse equivalences, but is never coarsely discontinuous if $G$ is infinite.  
\end{Nonexample}

If $X$ is a large scale space and $G$ is a group acting on $X$ by coarse equivalences, then we can consider the set of all coarse equivalences $X \rightarrow X$ which are close to the identity when considered as maps on $X_G$. If we identify in this set those maps which are close as maps on $X$, then we obtain a group under composition which we denote by $\mathsf{Aut}(X/ X_G)$. In the language of category theory, this is the automorphism group of $p_G: X \rightarrow X_G$ in the slice category $(\mathbf{Coarse}/\sim)/X_G$. If $X$ satisfies a connectedness condition and $G$ acts coarsely discontinuously, then we will see that $\mathsf{Aut}(X/ X_G)$ is naturally isomorphic to $G$. 

If $\mathcal{U}$ is a cover of a set $X$, then the \textbf{$\mathcal{U}$-component of $x \in X$} is the set of all $x' \in X$ for which there is a finite sequence $(U_i)_{0\leq i \leq n}$ of elements of $\mathcal{U}$ with $x \in U_0$, $x' \in U_n$ and $U_i \cap U_{i+1} \neq \varnothing$ for $0 \leq i < n$. A set is called \textbf{$\mathcal{U}$-connected} if it contains at most one $\mathcal{U}$-component.

\begin{Definition}\label{oneended}
Let $X$ be a large scale space and let $\mathcal{U}$ be a uniformly bounded cover of $X$. We say that $X$ is \textbf{coarsely one-ended at scale $\mathcal{U}$} if for every bounded set $K$ in $X$ there is a bounded set $K'$ containing $K$ such that $X \setminus K'$ is $\mathcal{U}$-connected. We say that $X$ is coarsely one-ended if it is coarsely one ended at some scale.
\end{Definition}

\begin{Proposition}
	Let $X$ be a proper geodesic metric space. Then $X$ is coarsely one-ended if and only if it is topologically one-ended, that is, for every bounded set $K\subseteq X$ there is a bounded set $K'\subseteq X$ so that $K\subseteq K'$ and $X\setminus K'$ is topologically connected. 
\end{Proposition}

\begin{proof}
	$\left( \Rightarrow \right) :$ Suppose $X$ is coarsely one-ended at scale $\mathcal{U}$. Without loss of generality, let $\mathcal{U}$ be the cover by $R$-balls, $R > 0$. Let $K$ be a bounded set and consider $N=\st(K,\mathcal{U})$. Then there is a $L$ containing $N$ so that $N\subseteq L$ and $X\setminus L$ is $\mathcal{U}$-connected. If $X\setminus L$ is topologically connected, then we're done. Otherwise, let $\{C_i\}_{i\in I}$ be the connected components of $X\setminus L$. Since $X\setminus L$ has one $\mathcal{U}$-component, for each connected component $C_i$ there is a distinct connected component $C_j$ and points $x_i\in C_i$ and $x_j\in C_j$ so that $x_i \mathcal{U} x_j$. In particular, $d(x_i,x_j)<2R$. Let $\gamma_{i,j}$ be a geodesic from $x_i$ to $x_j$. Then $\gamma_{i,j}$ must intersect $L$ (lest $C_i$ and $C_j$ not be distinct connected components) but $\gamma_{i,j}$ does not intersect $K$ since $N$ is an $2R$-ball about $K$ and the length of $\gamma_{i,j}$ is at most $2R$. In the new set $\gamma_{i,j} \cup X\setminus L$, $x_i$ and $x_j$ are in the same connected component. It follows from a Zorn's lemma argument that we can add a union of geodesics, none of which intersect $K$, to $X \setminus L$ to obtain a connected subspace. 
	
	$\left( \Leftarrow \right) :$ This follows from the observation that if a space is topologically connected, then it is $\mathcal{U}$-connected for any open cover $\mathcal{U}$.
\end{proof}

\begin{Lemma}\label{useful}
Let $(X,\mathcal{X})$ be an unbounded large scale space, and let $G$ be a group that acts on $X$ by coarse equivalences. Then the action of $G$ is coarsely discontinuous if and only if for every finite subset $F \subseteq G$ and every pair of uniformly bounded familes $\mathcal{U}$ and $\mathcal{V}$ in $X$, there is a bounded subset $K$ such that for any $x,y \in X \setminus K$ with $x \mathcal{U} y$, and any $g_1 \neq g_2$ in $F$, $\{ \{ g_1 \cdot x, g_2 \cdot y\} \}$ does not refine $\mathcal{V}$.
\end{Lemma}
\begin{proof}
Suppose the action is coarsely discontinuous, and let $F \subseteq G$ be a finite subset and $\mathcal{U}$ be a uniformly bounded family. For each $g \in F \cdot F^{-1} \setminus \{e\}$ choose a bounded subset $K_g$ such that if $x \notin K_g$ then  $\{ \{x, g\cdot x\} \}$ does not refine $\st(\mathcal{V}, \mathcal{U}')$, where
$$
\mathcal{U}' = \bigcup_{g \in F} g \cdot \mathcal{U}.
$$
Define the bounded set  
$$
K = \bigcup_{h \in F^{-1}} \bigcup_{g \in F \cdot F^{-1}} h\cdot K_g.
$$
If $x \mathcal{U} y$, $x, y \notin K$ and $\{ g_1 \cdot x, g_2 \cdot y \} \subseteq V \in \mathcal{V}$ with $g_1 \neq g_2$, then $(g_2 \cdot y) \mathcal{U}' (g_2 \cdot x)$ so $(g_1 \cdot x) \st(\mathcal{V}, \mathcal{U}') (g_2 \cdot x)$. But $g_2 x = g_2 g_1^{-1} g_1 \cdot x$ and $g_1\cdot  x \notin K_{g_2 g_1^{-1}}$, a contradiction. The converse is easy to check.
\end{proof}

\begin{Theorem}
	Let $(X,\mathcal{X})$ be an unbounded large scale space, and let G be a group that acts on $X$ coarsely discontinuously by coarse equivalences. If $(X,\mathcal{X})$ is coarsely one-ended, then there is a canonical group isomorphism $G\cong \mathsf{Aut}(X/X_G)$.
\end{Theorem}

\begin{proof}
Suppose $(X, \mathcal{X})$ is coarsely one-ended at scale $\mathcal{U}$. Define a map $\Phi$ from $G$ to $\mathsf{Aut}(X/X_G)$ by sending $g$ to its action on $X$. This is clearly a group homomorphism. We claim it is surjective. Let $f\in \mathsf{Aut}(X/X_G)$ and suppose (by use of Lemma \ref{XGUF}) that $f$ is close to the identity on $X_G$ as witnessed by $\st(\mathcal{V},\mathcal{F})$ with $\mathcal{V}$ a cover in $\mathcal{X}$ (which we may assume coarsens $\mathcal{U}$) and $\mathcal{F}=\{\{F\cdot x\}\}_{x\in X}$ for $F\subseteq G$ finite. Then for every $x \in X$ there is a $f_1 f_2^{-1} \in F \cdot F^{-1}$ and an $x'$ such that $f(x) = f_1 f_2^{-1} \cdot x'$ and $x \mathcal{V} x'$. Hence for every $x$ there is a $f_1 f_2^{-1} \in F \cdot F^{-1}$ such that $(f_1 f_2^{-1} \cdot x)\mathcal{V}' f(x)$, where $\mathcal{V}'$ is the bounded family
$$
\mathcal{V}' = \bigcup_{h \in F \cdot F^{-1}} h\cdot \mathcal{V}.
$$
For each $x$, pick such an $f_1 f_2^{-1} \in F \cdot F^{-1}$ and call it $\alpha(x)$. We claim that outside of a bounded set, $\alpha(x)$ is uniquely defined. Using Lemma \ref{useful}, choose a bounded set $K$ such that for any $x,y\notin K$ with $x \mathcal{V} y$ and any $g_1 \neq g_2$ in $F \cdot F^{-1}$,  $\{g_1 \cdot x, g_2 \cdot y\}$ does not refine $\st(f(\mathcal{U}), \mathcal{V}')$. If $x \mathcal{U} y$ with $x,y \notin K$, then $(\alpha(x)\cdot x)\mathcal{V}'f(x) f(\mathcal{U}) f(y) \mathcal{V}' (\alpha(y)\cdot y)$, so we must have $\alpha(x) = \alpha(y)$. Choose a bounded set $K'$ containing $K$ such that $X \setminus K'$ is $\mathcal{U}$-connected. Then for any two points $x,y \notin K'$, $x$ and $y$ are connected by a chain of elements of $\mathcal{U}$ and so $\alpha(x) = \alpha(y)$ by the above. Thus, $f$ is close to the action of some $h \in F \cdot F^{-1}$ outside of a bounded set, as witnessed by $\mathcal{V}'$. Since $f$ is large scale continuous, it must be close to the action of $h$ on all of $X$, which shows that $\Phi$ is surjective. By coarse discontinuity, $\Phi$ is also injective, since no action of $g$ is close to the identity on $(X, \mathcal{X})$ so long as $X$ is unbounded.
\end{proof}

Both coarse one-endedness and coarse discontinuity are necessary for the above theorem to hold true. To see the former, take $X$ to be the subspace of $\mathbb{R}^3$ consisting of the positive $x$, $y$ and $z$-axes. Let $G = \mathbb{Z}/3\mathbb{Z}$ act on this space via $1\cdot (x,y,z) = (z,x,y)$. The space $X_G$ is then coarsely equivalent to the positive real axis. But any permutation of $x$,$y$ and $z$ gives rise to an element of $\mathsf{Aut}(X/X_G)$. For coarse discontinuity, take for example the action of a finitely generated group $G$ on itself by right translation (see Example \ref{exgroup}). If we denote the underlying large scale space of $G$ by $|G|$ then $|G|_G$, with respect to this action, is just $|G|$ (see Nonexample \ref{Gright}), and so $\mathsf{Aut}(|G|/ |G|_G)$ is trivial regardless of what $G$ is.

\section{The maximal Roe algebra}
We recall the definition of the maximal Roe algebra from \cite{GongMaximal} (see also \cite{OyonoOyonoYu}). For the remainder of this section, $X$ denotes a discrete bounded geometry metric space (for example, a finitely generated group with a word metric). Recall that a metric space has \textbf{bounded geometry} if for every $R > 0$ there is an integer $N$ such that every $R$-ball in $X$ has at most $N$ elements. Our goal is to relate the Roe algebra of $X$ with that of $X_G$ for a coarsely discontinuous action of $G$, where $G$ is a countable group. 

Fix a separable infinite-dimensional Hilbert space $H$ and consider the algebra of bounded operators on $\ell^{2}(X) \otimes H$. We can view an operator $T$ on $\ell^{2}(X) \otimes H$ as a matrix $(T_{x,y})_{(x,y) \in X \times X}$ of operators on $H$. We say that $T$ has \textbf{propagation less than $R$} if $T_{x,y} = 0$ for $d(x,y) \geq R$. The \textbf{support} of $T$ is the subset of $X \times X$ for which $T_{x,y} \neq 0$. Denote by $C[X]$ the algebra of all bounded operators $T$ on $\ell^2(X) \otimes H$ such that, when $T$ is written as a matrix $(T_{x,y})_{x,y \in X}$ of operators,
\begin{itemize}
\item $T_{x,y}$ is compact for all $x$ and $y$ (that is, $T$ is \textbf{locally compact}), and
\item there exists an $R> 0$ such that $T$ has propagation less than $R$ (that is, $T$ has \textbf{finite propagation}). 
\end{itemize}
The (usual) \textbf{Roe algebra} $C^\ast(X)$ of $X$ is the operator norm closure of $C[X]$ in $B(\ell^2(X) \otimes H)$. However, for our purposes we will need a different $C^\ast$-algebra: the maximal Roe algebra. 

\begin{Definition}[\cite{GongMaximal}]
The \textbf{maximal Roe algebra} $C^\ast_{\mathsf{max}}(X)$ of $C[X]$ is the completion of $X$ with respect to the the $\ast$-norm 
$$
||T|| = \mathsf{sup}_{(\phi, H_\phi)} ||\phi(T) ||_{B(H_\phi)}
$$
where $(\phi, H_\phi)$ runs through representations $\phi$ of $C[X]$ on a Hilbert space $H_\phi$. 
\end{Definition}

Note that it follows from a ``partial translation decomposition'' argument (see \cite{GongMaximal}) that this norm is well-defined. If $G$ is a countable group and $X$ is a metric space then the large scale structure on $X_G$ is induced by a metric (which we may assume is discrete) since it is countably generated. It is easy to see that $C[X_G]$ (and thus $C^\ast(X_G)$ and $C^\ast_{\max}(X_G)$) is the same for any two metrics inducing the same large scale structure, so from now on we will assume that some metric on $X_G$ has been chosen which induces the large scale structure. It follows from Lemma \ref{XGUF} that if $X$ has bounded geometry, then so does $X_G$. 

Clearly $C[X]$ contains all the rank one operators 
$$
e_{(x,v), (y, w)}: \delta_z \otimes u \quad \mapsto  \quad \langle \delta_y  \otimes w, \delta_z \otimes u \rangle \delta_x \otimes v.
$$
It follows that the maximal Roe algebra contains a closed two-sided ideal canonically isomorphic to $\mathcal{K}$, the compact operators on $\ell^2(X) \otimes H$ (this is because $\mathcal{K}$ is the universal $C^\ast$-algebra generated by a system of matrix units). We will want to work with the quotient $
C^\ast_{\max}(X) / \mathcal{K}.$

Suppose a group $G$ acts on $X$ by coarse equivalences. For an element $g \in G$, let $M_g$ be the operator on $\ell^2(X) \otimes H$ given by 
$$
(M_g)_{x,y} = \begin{cases}
1_H & \text{ if } gy = x \\
0 & \text{ otherwise }
\end{cases}
$$
Note that $M_{g} M_h = M_{gh}$ and $M_g^\ast = M_{g^{-1}}$. Thus $G$ has an induced action on $B(\ell^2(X) \otimes H)$ via
$$
g \cdot T = M_g T M_g^\ast.
$$
This action restricts to an action on $C[X]$ which extends to an action on $C_{\max}^\ast(X)$. This action preserves the ideal $\mathcal{K}$, so we obtain an action on $C^\ast_{\max}(X) / \mathcal{K}$. Note also that $M_g$, while not an element of $C[X_G]$ (it is not locally compact), is a multiplier of $C[X_G]$ because it has finite propagation.

\begin{Theorem}\label{RoeXG}
Let $G$ be a countable group which acts coarsely discontinuously by coarse equivalences on a discrete bounded geometry metric space $X$ . Let $\alpha$ be the action of $G$ on $C^\ast_{\max}(X)/\mathcal{K}$ given by $g\cdot [T] = [M_g T M_g^\ast]$. Then there is a canonical $\ast$-isomorphism
$$
\xymatrix{
\left( C^\ast_{\max}(X) / \mathcal{K} \right) \rtimes_{\alpha} G
\ar[r]^(.55){\cong} & 
C^\ast_{\max}(X_G) / \mathcal{K}
}
$$
where $\rtimes_\alpha$ denotes the full crossed product. 
\end{Theorem}
\begin{proof}
Since for any $T \in C[X]$,
$||T M_g ||^2 = ||TM_g (TM_g)^\ast|| = ||TT^\ast|| = ||T||^2$
in $C^\ast_{\max}(X)$, the map $T \mapsto T M_g$ extends to a multiplier of $C^\ast_{\max}(X_G)$. Define a map $\Phi$ from $C_c[G,C^\ast_{\max}(X)]$ to  $C^\ast_{\max}(X_G)$ by $T\delta_g \mapsto  T M_g$. One checks that this actually defines a $\ast$-homomorphism, where the product on $C_c[G, C^\ast_{\max}(X)]$ is the convolution product with respect to the action $\alpha$. This gives rise to a $\ast$-homomorphism
$$
\Phi: C^\ast_{\max}(X) \rtimes_\alpha  G \rightarrow C^\ast_{\max}(X_G)
$$
Since the image of $\mathcal{K} \rtimes_\alpha G$ under this map is contained in $\mathcal{K}(\ell^{2}(X) \otimes H)$, we get a map
$$
\Phi'': C^\ast_{\max}(X)/ \mathcal{K} \rtimes_\alpha  G
 \rightarrow 
 C^\ast_{\max}(X_G)/ \mathcal{K}
$$
Here we are implictly using the fact that the full crossed product functor is exact, so that 
$$
C^\ast_{\max}(X)/ \mathcal{K} \rtimes_\alpha  G \cong 
\left( C^\ast_{\max}(X) \rtimes_\alpha G \right) / \left( \mathcal{K} \rtimes_\alpha  G \right).
$$
We claim that $\Phi''$ is a $\ast$-isomorphism. We will prove this by constructing an inverse $\Psi''$ to it. 

Let $T \in C[X_G]$. We claim that $T$ can be written as
$$
T = \sum_{g \in G} T_g M_g
$$
where each of the $T_g$ are in $C[X]$ and only finitely many of the $T_g$ are non-zero. Moreover, we claim that if $T = \sum_{g \in G} S_g M_g $
is some other decomposition of $T$ with each $S_g \in C[X]$, then $S_g - T_g \in \mathcal{K}(\ell^{2}(X) \otimes H)$ for every $g \in G$. To prove the existence of such a decomposition, note that by Lemma~\ref{XGUF}, there exists an $R>0$ and a finite set $F= \{g_0, \ldots, g_k\}$ of elements of $G$ such that the support of $T$ can be written as a disjoint union $\sqcup_{i=0}^{k} R_{g_i}$
where for any $(x,y) \in R_{g_i}$, there exists an $x'$ such that $g_ix' = y$ and $d_X(x,x') \leq R$ (note that this distance is in $X$, not $X_G$). Write the corresponding decomposition of $T$ (thinking of $T$ as a map from $X\times X$ to $\mathcal{K}(H)$) as
$$
T = \sum_{i=0}^{k} T|_{R_{g_i}}
$$
It follows from the definition of $R_{g_i}$ that $M_{g_i}^\ast T|_{R_{g_i}} $ (and thus also $ M_{g_i} M_{g_i}^\ast (T|_{R_{g_i}}) M_{g_i}^\ast = (T|_{R_{g_i}}) M_{g_i}^\ast$) is an element of $C[X]$. Thus we have our decomposition
$$
T = \sum_{i=0}^k T_i  M_{g_i},
$$
where $T_i = T|_{R_{g_i}} M_{g_i}^\ast$. 

Now suppose there is another such decomposition $T = \sum_{i=0}^l S_i M_{g_i}$ with each $S_i \in C[X]$. Then $\sum_{i=0}^l (S_i-T_i) M_{g_i} = 0$ with $T_i = 0$ for $i > k$ for convenience. Pick an $R' > 0$ such that for every $i$, $S_i - T_i$ has propagation less than $R'$. Using Lemma \ref{useful}, choose a bounded set $K$ such that for every pair $g_i \neq g_j$ in $\{g_0, \ldots ,g_l\}$, $d(g_i\cdot x, g_j \cdot x) > 2R'$. Notice that for any $i$, $0 \leq i \leq l$, $(S_i-T_i) M_{g_i}$ can only have a nonzero $(x,y)$ entry if $d(g_i \cdot x, y) \leq R'$. If $d(g_j \cdot x, y) \leq R'$ for some other $0 \leq j \leq l$, then 
$$
d(g_i \cdot x, g_j \cdot x) \leq 2R'
$$
implies $x \in K$. Thus for $x \notin K$, the $(x,y)$ entry of the sum $\sum_{i=0}^l (S_i-T_i) M_{g_i}$ is contributed to by exactly one $(S_i-T_i) M_{g_i}$. It follows that the support of every $(S_i- T_i)$ is a finite set, and thus each $(S_i - T_i)$ is compact.

We are now ready to define our inverse map. For any $T \in C[X_G]$, decompose $T$ as $T = \sum_{g \in G} T_g M_g$
with each $T_g \in C[X]$
and define 
$$\Psi(T) = \sum [T_g] \delta_g \in C^\ast_{\max}(X)/ \mathcal{K} \rtimes_\alpha  G.$$
Our previous calculations show that this is well-defined. One easily checks that this defines a $\ast$-homomorphism, and so it extends to a map
$$
\Psi': C^\ast_{\max}(X_G) \rightarrow C^\ast_{\max}(X)/ \mathcal{K} \rtimes_\alpha  G
$$
The final step is to note that the image of $\mathcal{K}$ under this map is $0$, so we have an induced map
$$
\Psi'': C^\ast_{\max}(X_G)/ \mathcal{K} \rightarrow
C^\ast_{\max}(X)/ \mathcal{K} \rtimes_\alpha  G
$$
which by construction is a two-sided inverse for $\Phi''$.
\end{proof}

Another way to write the conclusion of Theorem \ref{RoeXG} above is that there is a short exact sequence
\begin{equation} \label{exactXG}
0\to \mathcal{K} \to C^\ast_{\max}(X_G) \to (C^\ast_{\max}(X) / \mathcal{K}) \rtimes_{\alpha} G \to 0.
\end{equation}
The proof of Theorem \ref{RoeXG} was inspired by the proof of Proposition 2.8 in \cite{OyonoOyonoYu}. We now show that we recover this result as a special case. Let $\Gamma$ be a residually finite finitely generated group and let 
$$
\Gamma_0 \supset \Gamma_1 \supset \ldots 
$$
be a sequence of normal finite index subgroups such that $\cap_{i \in \mathbb{N}} \Gamma_i = \{e\}$. Let $d_\Gamma$ be the left-invariant metric associated to some generating set in $\Gamma$ and give $\Gamma/\Gamma_i$ the metric $d_i$ defined by $d_i(a\Gamma_i, b\Gamma_i) = \min \{ d_\Gamma (a\gamma_1,b \gamma_2) \mid \gamma_1, \gamma_2 \in \Gamma_i\}$. We define the \textbf{box space} to be the set $X(\Gamma) = \sqcup_{i \in \mathbb{N}} \Gamma/\Gamma_i$ equipped with a metric $d$ such that 
\begin{itemize}
\item $d$ agrees with the metric $d_i$ defined above on $\Gamma/\Gamma_i$, 
\item $d(\Gamma/\Gamma_i , \Gamma/\Gamma_j) > i + j$ if $i \neq j$, and
\item  the action of $\Gamma$ on $X(\Gamma)$ induced by left translation is an action by isometries.
\end{itemize}
Note that this last point is really a matter of appropriately defining $d(x, y)$ for $x \in \Gamma_i$ and $y \in \Gamma_j$ with $i \neq j$. 

\begin{Corollary}[Proposition 2.8 in \cite{OyonoOyonoYu}]\label{YuCorollary}
In the situation above, there is a short exact sequence
$$
0\to \mathcal{K} \to C^\ast_{\max}(X(\Gamma)) \to A_\Gamma \rtimes_{\alpha} \Gamma \to 0.
$$
where $A_\Gamma = \ell^{\infty}(X(\Gamma), \mathcal{K}(H)) / C_0(X(\Gamma), \mathcal{K}(H)) $ and the action of $\Gamma$ on $A_\Gamma$ is induced by the action of $\Gamma$ on $X(\Gamma)$ given by right translation, that is $g \cdot f(h \Gamma_i) = f(h g \Gamma_i)$.
\end{Corollary}
\begin{proof}
Let $X'$ be the underlying set of $X(\Gamma)$ equipped with the smallest coarsely connected large scale structure, that is, wherein uniformly bounded families are precisely those which contain finitely many non-singleton sets (since $X(\Gamma)$ is countable, this large scale structure is metrizable). Then $\Gamma$ acts on $X'$ by coarse equivalences. Moreover, this action is coarsely discontinuous. Indeed, since $\cap_{i \in \mathbb{N}} \Gamma_i = \{e\}$, every $\gamma \in \Gamma$ only fixes a finite number of points in $X'$. It is easy to check, by similar arguments as in Example \ref{exgroup},  that $X(\Gamma)$ is coarsely equivalent to $X'_\Gamma$ in the sense of Definition \ref{defXG}. Thus we can assume that $X'_\Gamma = X(\Gamma)$ as metric spaces. From Theorem \ref{RoeXG}, we have an exact sequence
$$
0\to \mathcal{K} \to C^\ast_{\max}(X(\Gamma)) \to (C^\ast_{\max}(X') / \mathcal{K}) \rtimes_{\alpha} \Gamma \to 0.
$$
It is enough then to show that $C^\ast_{\max}(X') /\mathcal{K}$ is $\ast$-isomorphic to $A_\Gamma$ in a way that preserves the action of $\Gamma$. There is an obvious $\ast$-homomorphism
$$
\Theta: A_\Gamma \rightarrow C^\ast_{\max}(X') /\mathcal{K}
$$
given by sending an element $f \in \ell^{\infty}(X(\Gamma), \mathcal{K}(H))$ to the (zero-propagation) diagonal matrix with entries $(f(\gamma))_{\gamma \in \Gamma}$. This map is clearly injective since such an $f$ represents a compact operator if and only if it is in  $C_0(X(\Gamma), \mathcal{K}(H))$. It remains to show it is surjective. Let $T \in C[X']$. Then by definition of the large scale structure on $X'$, $T$ can be written as $T' + T''$ where $T'$ has finitely many entries and $T''$ is a diagonal matrix of compact operators. Thus $T''$ is in the image of the map $\Theta$ above. Thus the image of $\Theta$ is dense and so $\Theta$ is surjective.
\end{proof}

\begin{Remark}
In fact, the action of $\Gamma$ on $X(\Gamma)$ in Proposition 2.8 in \cite{OyonoOyonoYu} is implied to be by \emph{left} translation. However, the authors believe that this is an error in \cite{OyonoOyonoYu}. The key observation is that if $d(e,g) \leq R$ for some $g \in \Gamma$ with $d$ being a left-invariant metric, then left translation by $g$ on $\ell^2(\Gamma)$ is not in general a finite propagation operator because $d(a, ga) = |a^{-1}ga|$. On the other hand \emph{right} translation by $g$ \emph{is} an operator of propagation less than $R$ since $d(a, ag) = |a^{-1}ag| \leq R$. A similar argument shows the same fact for the action of $g$ on $X(\Gamma)$. Thus in the proof of Proposition 2.8 in \cite{OyonoOyonoYu}, the operators $L_{\gamma_i \Gamma}$ should really be right translation operators and not left translation operators.
\end{Remark}

This corollary together with Remark 2.12 in \cite{OyonoOyonoYu} also shows  that there is no hope for an exact sequence of the form of (\ref{exactXG}) where the maximal Roe algebra is replaced by the usual Roe algebra and the full crossed product is replaced by the reduced crossed product. We will show, however, that in case $X$ has Yu's Property A (first introduced in \cite{YuEmbed}) and the group $G$ acting on it is amenable, we can make the replacement. 

Many equivalent definitions of Property A exist in the literature. We will use a definition (which is equivalent to Property A for bounded geometry discrete metric spaces) due to Dadarlat-Guentner. For an index set $S$, let $\Delta(S)$ denote the set of formal linear combinations 
$$
\sum_{s \in S} a_s \cdot s 
$$
such that $a_s \in [0,1]$ for each $s$, $a_s = 0$ for all but finitely many $s$, and $\sum a_s = 1$. We will equip $\Delta(S)$ with the $l^1$ metric. The \textbf{star} of a vertex $s \in S$ is the set of all elements of $\Delta(S)$ with $a_s \neq 0$. By a \textbf{partition of unity} on a set $X$, we mean a map $\phi: X \rightarrow \Delta(S)$ for some set $S$. 

\begin{Definition}\cite{DaGu}
 A large scale space $X$ is \textbf{exact} if for each uniformly bounded cover
 $\mathcal{U}$ of $X$ and each $\varepsilon > 0$ there is a partition of unity
 $\phi:X\to \Delta(S)$ such that point-inverses of stars of vertices form a uniformly bounded cover of $X$
 and the mesh of $\phi(\mathcal{U})$ is smaller than $\varepsilon$.
\end{Definition}

Also recall from \cite{DaGu} that a bounded geometry discrete metric space is exact if and only if it has Property A.

\begin{Theorem}\label{XGpropertyA}
Let $G$ be a countable group which acts on a discrete bounded geometry metric space $X$ by coarse equivalences. If $G$ is amenable and $X$ has Property A, then $X_G$ has Property A. 
\end{Theorem}
\begin{proof}
We will prove that $X_G$ is exact. Let $\mathcal{X}$ be the large scale structure on $X$,  let $\st(\mathcal{U},\mathcal{F})$ be a uniformly bounded family in $X_G$, with $\mathcal{F} = \{ F \cdot x \mid x \in X \}$ and $\mathcal{U} \in \mathcal{X}$, and let $\varepsilon >0$. By the amenability of $G$, we have that there is a finite $E\subseteq G$ so that for all $g\in F\cdot F^{-1}$, 
	$$\frac{|E\Delta E\cdot g|}{|E|}<\varepsilon/3.$$
	Since G acts by coarse equivalences, we have that $g\cdot\mathcal{U}=\{g\cdot U\}_{U\in\mathcal{U}}$ is in $\mathcal{X}$ for all $g$ and hence  
	$$\mathcal{U}^E = \bigcup\limits_{k\in E}k\cdot\mathcal{U}$$
	 is also in $\mathcal{X}$.

	Since $X$ is exact, we have that there is a partition of unity $(\phi_i)_{i\in I}$ such that the family $\mathcal{V} = (V_i)_{i \in I}$ is in $\mathcal{X}$, where $V_i  = \{x \in X \mid \phi_i(x) \neq 0\}$, and for every $x,y\in U$ with $U\in\mathcal{U}^E$, we have 
	$$\sum\limits_{i\in I} |\phi_i(x)-\phi_i(y)|<\frac{\varepsilon}{3}.$$
	Define a new partition of unity $(\psi_i)_{i\in I}$ on $X$ via 
	$$\psi_i(x)=\frac{1}{|E|}\sum\limits_{k\in E} \phi_i(k\cdot x),$$
 and let $\mathcal{W} = (W_i)_{i \in I}$ be the cover of $X$ given by $W_i = \{x \in X \mid \psi_i(x) \neq 0\}$. We claim that $\mathcal{W}$ is uniformly bounded in $X_G$. Indeed, $x\in W_i$ implies $\psi_i(x)\neq 0$ so there is a $k\in E$ so that $k\cdot x\in V_i$. It follows that $\mathcal{W}$ refines the cover $\st(\mathcal{V}, \mathcal{E})$, where $\mathcal{E} = \bigcup_{k\in E} \{\{x,k\cdot x\} \mid x \in X \}$.
 
It remains to show that for any $x,y\in\st(U,\mathcal{F})$ with $U \in \mathcal{U}$, we have $\sum_{i\in I}|\psi_i(x)-\psi_i(y)|<\varepsilon$. It is enough to show that (1)
	 for any $x,y\in U$ we have $\sum_{i\in I}|\psi_i(x)-\psi_i(y)|<\varepsilon/3$ and (2) for $x\in X$ and $g,h\in F$ we have $\sum_{i\in I}|\psi_i(g\cdot x)-\psi_i(h\cdot x)|<\varepsilon/3$. 
	 
	 We first show inequality (1). Let $x,y\in U$ for some $U\in\bigcup\limits_{k\in E}k\cdot\mathcal{U}$. Then 
	 $$\sum\limits_{i\in I}|\psi_i(x)-\psi_i(y)|\leq \frac{1}{|E|}\sum\limits_{k\in E}\sum\limits_{i\in I}|\phi_i(k\cdot x)-\phi_i(k \cdot y)|.$$
	 For any $k\in E$, $x,y\in U$ implies $k\cdot x,k\cdot y\in k\cdot U \in \mathcal{U}^E$, so by the construction of the $\phi_i$, 
	 
	 $$\sum\limits_{i\in I}|\psi_i(x)-\psi_i(y)|\leq\frac{1}{|E|}\cdot |E|\cdot\frac{\varepsilon}{3}=\frac{\varepsilon}{3}.$$
	 We now show (2). Let $x\in X$ and $g,h\in F$. Then 
	 $$\sum\limits_{i\in I}|\psi_i(g\cdot x)-\psi_i(h\cdot x)|=\frac{1}{|E|}\sum\limits_{i\in I}|\sum\limits_{k\in E} \phi_i(k\cdot g\cdot x)-\phi_i(k\cdot h\cdot x)|.$$
	 Notice that if $k\cdot g\in E\cdot g\cap E\cdot h$, then the term $\phi_i(k\cdot g \cdot x)$ is cancelled out. So from the above we get
	 \begin{align*}
	 \sum\limits_{i\in I}|\psi_i(g\cdot x)-\psi_i(h\cdot x)|
	 & =\frac{1}{|E|}\sum\limits_{i\in I}|\sum\limits_{l\in E\cdot g\setminus E\cdot h}\phi_i(l\cdot x) \quad - \sum\limits_{m\in E\cdot h\setminus E\cdot g}\phi_i(m\cdot x)| \\
	 & \leq\frac{1}{|E|}\sum\limits_{l\in E\cdot g\Delta E\cdot h} \ \sum\limits_{i\in I}|\phi_i(l\cdot x)| \\
	 & = \frac{|E\cdot g\Delta E\cdot h|}{|E|}.
	 \end{align*}
	 since each $\phi_i$ is a partition of unity. But $|E\cdot g\Delta E\cdot h| = |E \Delta E\cdot h g^{-1}|$, so by the condition on $E$, we have 
	 $$
	 \frac{|E\cdot g\Delta E\cdot h|}{|E|} \leq \varepsilon/3
	 $$
	 as required.
\end{proof}

Note that by Proposition \ref{XGlight} in this paper and Corollary 9.4 in \cite{WeighillDydakMonotone}, $X$ has Property A if $X_G$ does, for any countable group $G$. The following corollary was already proved in \cite{DydakVirk}.

\begin{Corollary}
Let $G$ be a finite group which acts on a discrete bounded geometry metric space $X$ by coarse equivalences. If $X$ has Property A, then $X/G$ has Property A when endowed with the Hausdorff metric. 
\end{Corollary}
\begin{proof}
Since Property A is invariant under coarse equivalence, this follow from Propositions \ref{finiteG} and \ref{obsfinitemetric} and the fact that any finite group is amenable. 
\end{proof}

The following result is a special case of Corollary 5.6.17 in \cite{BrownOzawa}. See \cite{WilletSpakulaMaximal} for a direct proof.

\begin{Theorem}[\cite{BrownOzawa}]
If $X$ is a bounded geometry discrete metric space with Yu's Property A, then the canonical quotient $\lambda: C^\ast_{\mathsf{max}}(X) \rightarrow C^\ast(X)$ is a $\ast$-isomorphism.
\end{Theorem} 

Recalling that the full crossed product and reduced crossed product agree for amenable groups, we have the following corollary.

\begin{Corollary}
Let $X$ be a bounded geometry discrete metric space with Yu's Property A, and let $G$ be an countable amenable group acting on $X$ coarsely discontinuously by coarse equivalences. Then we have an exact sequence
$$
0\to \mathcal{K} \to C^\ast(X_G) \to (C^\ast(X) / \mathcal{K}) \rtimes_{r, {\alpha}} G \to 0.
$$
\end{Corollary}

By Proposition \ref{finiteG}, if $G$ is finite, $X_G$ is coarsely equivalent to $X/G$. This coarse equivalence gives rise to a (non-canonical) $\ast$-isomorphism $\phi: C^\ast_{\max}(X_G) \rightarrow C^\ast_{\max}(X/G)$. The map $\phi$ is constructed as follows: let $H = \oplus_{g \in G} H_g$ be an orthogonal decomposition of $H$ into infinite dimensional subspaces, and let $\psi_g: H \rightarrow H_g$ be a unitary isometry for every $g \in G$. For each equivalence class $[x] \in X/G$, choose a representative $s([x]) \in [x]$. Then we can define a unitary operator
$$
U_{p} : \ell^2(X_G) \otimes H \rightarrow \ell^2(X/G) \otimes H  
$$
by $U_{p}(\delta_x \otimes h) =  \delta_{[x]} \otimes \psi_g(h)$ where $x = g\cdot s([x])$. We can then define an isometry
$$
\mathsf{Ad} U_p: C^\ast[X_G] \rightarrow C^\ast[X/G]
$$
by $\mathsf{Ad} U_p (T) = U_p T U_p^\ast$, which extends to an isometry $\phi: C^\ast_{\max}(X_G) \rightarrow C^\ast_{\max}(X/G)$ as required. Since $\phi$ preserves all the rank-one projections, it also preserves the compact operators. Thus we obtain the following corollary of Theorem \ref{RoeXG}, Propositions \ref{quotientmetricls2} and \ref{finiteG}.

\begin{Corollary}
Let $G$ be a finite group acting coarsely discontinuously by coarse equivalences on a discrete bounded geometry metric space $X$, and let $X/G$ be the orbit space with the Hausdorff metric. Then there is a $\ast$-isomorphism
$$
\xymatrix{
\left( C^\ast_{\max}(X) / \mathcal{K} \right) \rtimes_{r, \alpha} G
\ar[r]^(.55){\cong} & 
C^\ast_{\max}(X/G) / \mathcal{K}
}
$$
Moreover, if $X$ has Property A, then the maximal Roe algebra above can be replaced by the usual Roe algebra.
\end{Corollary}

\begin{Example}
Let $G = \{1, \gamma\} \cong \mathbb{Z}/2\mathbb{Z}$ and let $G$ act on the metric space $X = \mathbb{Z}$ by $\gamma(x) = -x$. This action is coarsely discontinuous, and the quotient $X /G$ is coarsely equivalent to $\mathbb{N}$. Since $\mathbb{Z}$ has Property A and $\mathbb{Z}/2\mathbb{Z}$ is amenable, we have a $\ast$-isomorphism
$$
\xymatrix{
\left( C^\ast(\mathbb{Z}) / \mathcal{K} \right) \rtimes_{r} \mathbb{Z}/2\mathbb{Z}
\ar[r]^(.6){\cong} & 
C^\ast(\mathbb{N}) / \mathcal{K}
}
$$
\end{Example}

\begin{Example}
Recall from Remark \ref{XGrem1} that if $G$ is a finitely generated group and $X$ is the underlying set of $G$ equipped with the smallest coarsely connected large scale structure, then $X_G$ is coarsely equivalent to $G$, where the action of $G$ is by right translation. By similar arguments to the proof of Corollary \ref{YuCorollary}, one can check that $C^\ast_{\max}(X)/\mathcal{K}$ is naturally isomorphic to $\ell^{\infty}(|G|, \mathcal{K}(H)) / C_0(|G|, \mathcal{K}(H))$ where $|G|$ is the underlying set of $G$. It follows that there is a natural isomorphism
$$
C^\ast_{\max}(G)/\mathcal{K} \cong \ell^{\infty}(|G|, \mathcal{K}(H)) / C_0(|G|, \mathcal{K}(H)) \rtimes G.
$$
where the action of $G$ on $\ell^{\infty}(|G|, \mathcal{K}(H)) / C_0(|G|, \mathcal{K}(H))$ is given by right translation: $g \cdot f(h) = f(hg^{-1})$.
Compare this result to Theorem 4.28 in \cite{Roe} which states that 
$$
C^\ast_u(G)  \cong \ell^\infty(|G|) \rtimes_r G
$$
where $C^\ast_u(-)$ denotes the uniform Roe algebra.
\end{Example}

\section{Weak coarse quotient maps} \label{Sectypesofquotients}
In this section, we introduce the notion of weak coarse quotient map, of which $p_G: X \rightarrow X_G$ is an example when $G$ is finitely generated, and motivate the definition using the coarse category.

Recall that given a topological space $X$ and a surjective set map $f$ from the underlying set of $X$ to a set $Y$, the quotient topology on $Y$ is the finest topology that makes $f$ continuous. We now introduce an analogous notion for large scale spaces, based on the definition of quotient coarse structure in \cite{DikranjanZava}.

\begin{Definition}
Let $X$ be a large scale space and let $f$ be a surjective set map from the underlying set of $X$ to a set $Y$. Then the \textbf{quotient large scale structure on $Y$} is defined to be $\overline{f(\mathcal{X})}$ where $\mathcal{X}$ is the large scale structure on $X$ and $f(\mathcal{X})$ is the collection $\{ f(\mathcal{U}) \mid \mathcal{U} \in \mathcal{X} \}$. 
\end{Definition}

Clearly the quotient large scale structure is the smallest large scale structure which makes the map $f$ large scale continuous. The quotient large scale structure also has a universal property. In fact, the existence of the quotient large scale structure and the universal property follow from general categorical considerations in \cite{DikranjanZava}. For completeness, we present a direct proof here.

\begin{Proposition}\label{strictuniversalQuotient}
Let $f: X \rightarrow Y$ be a surjective large scale continuous map. Then $Y$ has the quotient large scale structure with respect to $f$ if and only $f$ satisfies the following universal property:
\begin{itemize}
\item[$\mathsf{(Q1)}$] for any large scale continuous map $g: X \rightarrow Z$ which is constant on the fibres of $f$, there is a unique large scale continuous map $h$ from $Y$ to $Z$ such that $hf = g$. 
\end{itemize}
\end{Proposition}
\begin{proof}
$(\Rightarrow)$ Let $\mathcal{X}$ be the large scale structure on $X$. The map $h$ is uniquely defined: $h(f(x)) = g(x)$. The fact that $h$ is large scale continuous follows from Lemma \ref{generatingMap}: 
$$
h(\overline{f(\mathcal{X})}) \subseteq \overline{hf(\mathcal{X})} = \overline{g(\mathcal{X})} \subseteq \mathcal{Z}
$$
where $\mathcal{Z}$ is the large scale structure on $Z$.

$(\Leftarrow)$ Suppose $\mathsf{(Q1)}$ holds. Consider the map $f': X \rightarrow Y'$, where $Y'$ is the underlying set of $Y$ equipped with the quotient large scale structure and $f'$ is the same as $f$ at the level of underlying sets. It follows that the identity set map $Y \rightarrow Y'$ must be large scale continuous, and the result follows from this.
\end{proof}

We may be tempted to define a weak coarse quotient map as a surjective large scale continuous map $f: X \rightarrow Y$ such that $Y$ has the quotient large scale structure with respect to $f$. The problem with this idea is that such a definition is not very ``coarse''. Indeed, we should expect a class of maps $\mathcal{E}$ defined by a large scale property to satisfy the following conditions:
\begin{itemize}
\item[$\mathsf{(LS1)}$] if $f$ is close to $g$, and $f$ is in $\mathcal{E}$, then so is $g$;
\item[$\mathsf{(LS2)}$] if $f$ is a large scale continuous map, and $\phi$ and $\psi$ are coarse equivalences such that the composite $\phi f \psi$ is defined, then $f$ is in $\mathcal{E}$ if and only if $\phi f \psi$ is in $\mathcal{E}$.
\end{itemize}
In fact, as the following lemma shows, $\mathsf{(LS2)}$ implies $\mathsf{(LS1)}$.

\begin{Proposition}
If a class $\mathcal{E}$ of large scale continuous maps satisfies $\mathsf{(LS2)}$ then it also satisfies $\mathsf{(LS1)}$.
\end{Proposition}
\begin{proof}
Let $f: X \rightarrow Y$ be in $\mathcal{E}$, and suppose $g$ is a map whose closeness to $f$ is witnessed by the uniformly bounded cover $\mathcal{U}$. Let $X'$ be the subspace of the product $X \times Y$ given by
$$
\{ (x,y) \mid f(x)\mathcal{U} y \}.
$$
The map $i: X \rightarrow X'$ given by $x \mapsto (x, f(x))$ is a coarse equivalence, and we have $\pi_2 \circ i = f$, where $\pi_2: X' \rightarrow Y$ is the projection onto the second coordinate. It follows that $\pi_2$ is in $\mathcal{E}$ by $\mathsf{(LS2)}$. We have a map $j: X \rightarrow X'$ given by $x \mapsto (x, g(x))$, which is also a coarse equivalence, such that $\pi_2 \circ j = g$. Applying condition $\mathsf{(LS2)}$, we have that $g$ is in $\mathcal{E}$ as well.
\end{proof}

The class of all surjective maps $f$ whose codomain carries the quotient large scale structure with respect to $f$ satisfies neither $\mathsf{(LS1)}$ nor $\mathsf{(LS2)}$. We thus introduce the following definition of weak coarse quotient map instead.

\begin{Definition}
Let $f: X \rightarrow Y$ be a large scale continuous map. Then $f$ is a \textbf{weak coarse quotient map} if it is coarsely surjective and there exists a uniformly bounded cover $\mathcal{V}$ of $Y$ such that the large scale structure on $Y$ is generated by $f(\mathcal{X}) \cup \{\mathcal{V}\}$, where $\mathcal{X}$ is the large scale structure on $X$. A cover $\mathcal{V}$ satisfying this property is called a \textbf{quotient scale} of $f$. 
\end{Definition}

\begin{Observation}\label{obsS}
Consider a group action $G$ on a large scale space $X$. If the group $G$ is generated by a finite set $S$, then the collection of all families 
$$
 \{ \{x, gx\} \mid x \in X \}
$$
is contained in the large scale structure generated by the single family
$$
 \mathcal{S} = \{ S \cdot x \mid x \in X \},
$$
so that the identity set map $X \rightarrow X_G$ is a weak coarse quotient map with quotient scale $\mathcal{S}$. 
\end{Observation}

Note that if $\mathcal{V}$ is a quotient scale for the weak coarse quotient map $f: X \rightarrow Y$, then so is any uniformly bounded coarsening of $\mathcal{V}$. In particular, we can always pick a quotient scale $\mathcal{V}$ such that $Y \subseteq \st(f(X), \mathcal{V})$. A weak coarse quotient map satisfies a universal property.

\begin{Proposition} \label{universalQuotient}
Let $f: X \rightarrow Y$ be a large scale continuous map and let $\mathcal{V}$ be a uniformly bounded cover of $Y$. Then the following are equivalent.
\begin{itemize}
\item[(a)] $f$ is a weak coarse quotient map with quotient scale $\mathcal{V}$,
\item[(b)] for any large scale continuous map $g: X \rightarrow Z$ such that $g(f^{-1}(\mathcal{V}))$ is uniformly bounded, there exists a unique-up-to-closeness map $h: Y \rightarrow Z$ such that $hf$ is close to $g$.
\end{itemize}
\end{Proposition}
\begin{proof}
(a) $\Rightarrow$ (b): Without loss of generality, choose the weak coarse quotient scale $\mathcal{V}$ so that $Y\subseteq \st(f(X),\mathcal{V})$. Suppose there is a large scale continuous $g:X\to Z$ so that $g( f^{-1}(\mathcal{V}))\in\mathcal{Z}$ where $\mathcal{Z}$ is the large scale structure on $Z$. 
		We define a map $h:Y\to Z$. Let $y\in Y$. Then we can pick a $V\in\mathcal{V}$ and an $x \in X$ so that $y\in V$ and $f(x)\in  V$. Define $h(y)=g(x)$. We claim that $h\circ f$ and $g$ are close.
		Indeed, let $x\in X$. Then $hf(x) = g(x')$ for some $x'$ such that $f(x') \mathcal{V} f(x)$. It follows that $g(f^{-1}(\mathcal{V}))$ witnesses the closeness of $g$ and $h \circ f$. The uniqueness up to closeness follows from the fact that $f$ is an epimorphism in the coarse category. Finally, if $\mathcal{U}$ is a uniformly bounded family in $X$, then $hf(\mathcal{U})$ refines $\st(g(\mathcal{U}), g(f^{-1}(\mathcal{V})))$. This, together with the fact that $\mathcal{V}$ and $f(\mathcal{X})$ together generate the large scale structure on $Y$, give that $h$ is large scale continuous.
		
		(b) $\Rightarrow$ (a): It is easy to check that $f$ must be an epimorphism, and hence coarsely surjective. Let $Y'$ be the underlying set of $Y$ with the large scale structure generated by $\mathcal{V}$ and $f(\mathcal{X})$. By hypothesis, there is a large scale continuous map $Y \rightarrow Y'$, which must be close to the identity on $f(X)$ and thus also on all of $Y$. Since the identity map is close to a large scale continuous map, it is itself large scale continuous, and it follows easily that $Y'$ and $Y$ have the same large scale structure.
		\end{proof}

Recall that for a category $\mathbb{C}$ and two morphisms $f,g: X \rightarrow Y$, a \textbf{coequalizer} of $f$ and $g$ is a morphism $h: Y \rightarrow Z$ such that $hf = hg$ and such that if $h': Y \rightarrow Z'$ is another morphism such that $h'f = h'g$, then there exists a unique morphism $i: Z \rightarrow Z'$ such that $ih = h'$. A \textbf{regular epimorphism} is a coequalizer of a pair of morphisms. In the category of topological spaces and continuous maps, the epimorphisms are the surjective continuous maps, which are not quotient maps in general. On the other hand, the regular epimorphisms are precisely the quotient maps. 

\begin{Proposition}
Let $f: X \rightarrow Y$ be a large scale continuous map. Then $[f]$ is a regular epimorphism in $\mathbf{Coarse}/\sim$ if and only if $f$ is a weak coarse quotient map.
\end{Proposition}
\begin{proof}
$(\Rightarrow)$: Suppose $f$ is the coequalizer of $[a],[b]: W \rightarrow X$, and let $\mathcal{V} = \{ \{f(a(w)), f(b(w)) \} \mid w \in W \}$. Then since $fa \sim fb$, $\mathcal{V}$ is uniformly bounded. We claim that $f$ is a weak coarse quotient map with quotient scale $\mathcal{V}$. Indeed, if $g: X \rightarrow Z$ is some other map that sends $f^{-1}(\mathcal{V})$ to a uniformly bounded family, then $ga \sim gb$ and so $[g]$ factors uniquely through $[f]$ in the coarse category. By Proposition \ref{universalQuotient}, $f$ is a weak coarse quotient map.

$(\Leftarrow)$: Suppose $f$ is a weak coarse quotient map with quotient scale $\mathcal{V}$. Let $W$ be the subspace of $X \times X$ given by $\{ (x,x') \mid f(x) \mathcal{V} f(x') \}$, and let $\pi_1$, $\pi_2$ be the projections $W \rightarrow X$, which are clearly large scale continuous. Then $f \pi_1 \sim f \pi_2$, and if $g \pi_1 \sim g \pi_2$, then $g$ sends $f^{-1}(\mathcal{V})$ to a uniformly bounded family, so that $[g]$ factors uniquely through $[f]$ in the coarse category by Proposition \ref{universalQuotient}. It follows that $[f]$ is the coequalizer of $[\pi_1]$ and $[\pi_2]$.
\end{proof}

\begin{Corollary}
Any coarse equivalence is a weak coarse quotient map. Moreover, the class of weak coarse quotient maps satisfies $\mathsf{(LS1)}$ and $\mathsf{(LS2)}$.
\end{Corollary}
\begin{proof}
Any isomorphism is a regular epimorphism, and regular epimorphisms are closed under composition with isomorphisms, which gives $\mathsf{(LS2)}$.
\end{proof}

Clearly if $f: X \rightarrow Y$ is a surjective large scale continuous map and $Y$ has the quotient large scale structure, then $f$ is a weak coarse quotient map with any uniformly bounded cover of $Y$ as a quotient scale. Much more general situations are possible, however, as the proposition below shows.

\begin{Proposition} \label{allmonogenic}
Let $Y$ be a large scale space. Then the following are equivalent:
\begin{itemize}
\item[(1)] $Y$ is monogenic, that is, the large scale structure on $Y$ is generated by a single uniformly bounded family $\mathcal{V}$;
\item[(2)] every coarsely surjective large scale continuous map $f: X \rightarrow Y$ is a weak coarse quotient map.
\end{itemize}
\end{Proposition}
\begin{proof}
(1)$\Rightarrow$(2): Pick the generating family $\mathcal{V}$ as the quotient scale.

(2)$\Rightarrow$(1): Let $Y'$ be the underlying set of $Y$ with the smallest large scale structure, and consider the identity set map $Y' \rightarrow Y$. The large scale structure on $Y'$ must be generated solely by a quotient scale $\mathcal{V}$ of this map, and the result follows.
\end{proof}

In particular, any coarsely surjective large scale continuous map whose codomain is a geodesic metric space is a weak coarse quotient map. Thus at least in the setting of geodesic metric spaces (or more generally, spaces which are coarsely equivalent to geodesic metric spaces -- see Proposition 2.57 in \cite{Roe}), weak coarse quotients are nothing but coarsely surjective maps. In light of this observation, one may ask whether we can find a smaller class of large scale continuous maps which includes the class of all surjective maps $f: X \rightarrow Y$ such that $Y$ has the quotient large scale structure and still satisfies $\mathsf{(LS1)}$ and $\mathsf{(LS2)}$. The following proposition shows that this is impossible.

\begin{Proposition}
Let $\mathcal{E}$ be the class of all surjective large scale continuous maps $f: X \rightarrow Y$ such that $Y$ has the quotient large scale structure. Then the class of weak coarse quotients maps is the smallest class of large scale continuous maps satisfying $\mathsf{(LS1)}$ and $\mathsf{(LS2)}$ and containing $\mathcal{E}$.
\end{Proposition}
\begin{proof}
Suppose $\mathcal{E}'$ is a class of  large scale continuous maps satisfying $\mathsf{(LS1)}$ and $\mathsf{(LS2)}$ and containing $\mathcal{E}$. We claim it contains all the weak coarse quotient maps. Let $f: X \rightarrow Y$ be a surjective large scale continuous map where the large scale structure on $Y$ is generated by $f(\mathcal{X})$ and the cover of subsets $\mathcal{V}$, where $\mathcal{X}$ is the large scale structure on $X$. Define $X'$ to be the subspace of the product $X \times Y$ given by 
$$
\{ (x,y) \mid f(x) \mathcal{V} y \}
$$
and let $i: X \rightarrow X'$ be the map $x \mapsto (x, f(x))$. One can check that $i$ is a coarse equivalence. The projection onto the second coordinate $\pi_2: X' \rightarrow Y$ is such that $\pi_2 \circ i = f$. The large scale structure on $Y$ is the quotient large scale structure with respect to $\pi_2$. Indeed, if $\mathcal{X}'$ is the large scale structure on $X'$, then $\pi_2(\mathcal{X}')$ clearly contains $f(\mathcal{X})$ by $\pi_2 \circ i = f$, as well as the cover $\mathcal{V}$ (take the image of the family $\Delta \times \mathcal{V}$ in $X \times Y$ restricted to $X'$, where $\Delta$ is the cover by singletons). Thus $\pi_2$ is in $\mathcal{E}$, and so $f$ is as well. Finally, we can weaken the requirement that $f$ be surjective to coarsely surjective by applying the above argument to the restriction $X \rightarrow f(X)$ of $f$ and then using $\mathsf{(LS2)}$ to compose with the inclusion $f(X) \rightarrow Y$, which is a coarse equivalence.
\end{proof}

We now make the connection with the notion of coarse quotient mapping in \cite{Zhang}. Recall from \cite{Zhang} that a map $f: X \rightarrow Y$ between metric spaces is called a \textbf{coarse quotient mapping with constant $K$} if it is large scale continuous and for every $\varepsilon$ there exists a $\delta = \delta(\varepsilon)$ such that for every $x \in X$
$$
B(f(x), \varepsilon) \subseteq f(B(x, \delta))^K
$$
where for $A \subseteq Y$, $A^L = \{ y \in Y \mid \exists_{a \in A} d(a, y) \leq L \}$. If $f: X \rightarrow Y$ is a coarse quotient mapping, then every uniformly bounded family $\mathcal{U}$ in $f(X)$ refines the image of $\st(f(\mathcal{V}), \mathcal{B}_K)$ for some uniformly bounded family $\mathcal{V}$ in $X$, where $\mathcal{B}_K$ is the cover of $Y$ by $R$-balls. Thus the restriction $f: X \rightarrow f(X)$ is a weak coarse quotient map with quotient scale $\mathcal{B}_K$. As noted in \cite{Zhang}, every coarse quotient mapping is coarsely surjective, so it follows that every coarse quotient mapping is a weak coarse quotient map. The converse is not true: simply take any large scale continuous and coarsely surjective map into a geodesic metric space which is not a coarse quotient mapping.

\begin{Proposition}
Let $X$ be a metric space and let $G$ be a finite group acting on $X$ by coarse equivalences. Then the identity set map $p_G: X \rightarrow X_G$ is a coarse quotient mapping in the sense of \cite{Zhang} for any metric inducing the large scale structure on $X_G$. 
\end{Proposition}
\begin{proof}
Let $K = \max \{ d(x, g\cdot x) \mid x \in X,\ g \in G \}$, which is finite by the definition of the large scale structure on $X_G$. Let $\varepsilon > 0$. From Lemma \ref{XGUF}, and using the fact that $G$ is finite, there is a uniformly bounded family $\mathcal{U}$ such that any ball $B(f(x), \varepsilon)$ in $X_G$ is contained in $\bigcup_{g \in G}\ g \cdot U$ for some $U \in \mathcal{U}$. It follows that $B(f(x), \varepsilon)$ is contained in $B(f(x), \delta)^K$ where $\delta = \mesh(\mathcal{U})$. 
\end{proof}

\section{Metrization of quotient large scale structures}\label{secmetrizable}
If $f: X\rightarrow Y$ is a weak coarse quotient map, and $X$ is a metric space, then the large scale structure on $Y$ is countably generated, hence metrizable. The following proposition gives an explicit construction of a metric on $Y$ which induces the large scale structure.

\begin{Proposition}\label{metricweakquotient}
Let $f: X \rightarrow Y$ be a weak coarse quotient map with quotient scale $\mathcal{V}$. Let $\mathcal{V}' = \st(\mathcal{V}, \mathcal{V})$. If $X$ is a metric space with metric $d_X$, then the large scale structure on $Y$ is induced by the metric $d_Y$ defined by $d_Y(y,y) = 0$ and for $y \neq y'$, 
$$
d_Y(y, y') = \mathsf{inf}\{ n  + \sum_{i=1}^{n} d_X(a_i, b_i) \mid f(a_1)\mathcal{V}y,\ f(b_n)\mathcal{V}y',\ f(b_i) \mathcal{V}' f(a_{i+1}),\ n\in \mathbb{Z}_{+} \}.
$$
\end{Proposition}
\begin{proof}
Let $\mathcal{Y}$ denote the large scale structure on $Y$. It is easy to check that $d_Y$ is a metric. Since $a\mathcal{V}b \implies d_Y(a,b) = 1$, the cover $\mathcal{V} \in \mathcal{Y}$ is uniformly bounded with respect to $\mathcal{V}$. The image under $f$ of any uniformly bounded family in $X$ is also clearly uniformly bounded with respect to $d_Y$. Thus since $\mathcal{Y}$ is generated by $f(\mathcal{X}) \cup \{\mathcal{V}\}$, we have the containment $\mathcal{Y} \subseteq \mathcal{L}(d_Y)$, where $\mathcal{L}(d_Y)$ is the large scale structure induced by $d_Y$. It remains to show that every element of $\mathcal{L}(d_Y)$ is an element of $\mathcal{Y}$. Let $\mathcal{U} \in \mathcal{L}(d_Y)$, and pick $M \in \mathbb{Z}$ such that $d_Y(y, y') < M$ for any $y\mathcal{U}y'$. By the definition of $d_Y$ we have that for every $y \mathcal{U} y'$ there is a sequence $(a_i, b_i)_{1 \leq i \leq k}$ of pairs of elements of $X$ such that 
\begin{itemize}
\item $k \leq M$ and $d(a_i, b_i) \leq M$ for all $i$,
\item $f(a_1)\mathcal{V}y,\ f(b_n)\mathcal{V}f(b)$ and $\ f(b_i) \mathcal{V}' f(a_{i+1})$ for all $i$. 
\end{itemize}
If $\mathcal{W} \in \mathcal{Y}$ is a common coarsening of $\mathcal{V}'$ and $f(\mathcal{B}_M)$, where $\mathcal{B}_M$ is the cover of $X$ by $M$-balls, it follows that $y$ is connected to $y'$ by a chain of at most $2M+1$ elements of $\mathcal{W}$, which shows that $\mathcal{U}$ is an element of $\mathcal{Y}$.
\end{proof}

\begin{Corollary}
Let $X$ be a metrice space, and let $G$ be a finitely generated group that acts on $X$ by coarse equivalences, with $G$ generated by the finite symmetric set $S$ containing the identity. Then the large scale structure on $X_G$ is induced by the metric defined by $d_{X_G}(x,x) = 0$ and for $x \neq x'$, 
$$
d_{X_G}(x, x') = \mathsf{inf}\{ n  + \sum_{i=1}^{n} d_X(a_i, b_i) \mid x \in  S^2 \cdot a_1,\ x' \in S^2 \cdot b_n,\ b_i \in S^4 \cdot a_{i+1},\ n\in \mathbb{Z}_{+} \}.
$$
where $S^n = \{ s_1 s_2 \cdots  s_n \mid \forall_i s_i \in S \}$.
\end{Corollary}
\begin{proof}
This follows from Observation \ref{obsS}.
\end{proof}

\begin{Corollary}\label{metricquotientls}
Let $X$ be a metric space and let $f$ be a surjective set map from the underlying set of $X$ to a set $Y$. Then the quotient large scale structure on $Y$ is induced by the metric $d'_f$ defined by $d'_f(y,y) = 0$ and for $y \neq y'$, 
$$
d'_f(y, y') = \mathsf{inf}\{ n  + \sum_{i=1}^{n} d_X(a_i, b_i) \mid f(a_1)=y,\ f(b_n)=y',\ f(b_i) = f(a_{i+1}),\ n\in \mathbb{Z}_{+} \}
$$
\end{Corollary}
\begin{proof}
The quotient large scale structure on $Y$ is the unique large scale structure for which $f$ is a weak coarse quotient map with quotient scale the cover by singletons. 
\end{proof}

The metric $d'_f$ in Corollary \ref{metricquotientls} may seem unfamiliar, but when $X$ is uniformly discrete, it coincides with the classical quotient metric. We briefly recall the definition of the quotient (pseudo)metric (see for example~\cite{Burago, Gromov07}). Let $X$ be a metric space and let $f: X \rightarrow Y$ be a map from the underlying set of $X$ to a set $Y$. Then the \textbf{quotient pseudometric} on $Y$ with respect to $f$ is defined to be 
$$
d_f(y, y') = \mathsf{inf}\{ \sum_{i=1}^{n} d_X(a_i, b_i) \mid  f(a_1) = y,\ f(b_n) = y',\ f(b_i) = f(a_{i+1}),\ n \in \mathbb{Z}_{+}  \}
$$
(in fact the definition is usually stated for an equivalence relation $E$ on $X$, but this is clearly the same thing as a surjective set map $X \rightarrow Y$). Note that this may not be a metric since distinct points may be distance $0$ apart. Recall the following definition (see for example~\cite{NowakYu}).

\begin{Definition}
A metric space $X$ is called \emph{uniformly discrete} if there is a constant $C > 0$ such that for any $x \neq x'$, $d(x, x') > C$.
\end{Definition}

\begin{Proposition}\label{quotientmetricls2}
Let $X$ be a uniformly discrete metric space and let $f$ be a surjective set map from the underlying set of $X$ to a set $Y$. Then the (classical) quotient pseudometric on $Y$ is a metric and induces the quotient large scale structure on $Y$ with respect to $f$.
\end{Proposition}
\begin{proof}
It is easy to see that the quotient pseudometric is a metric. Let $C>0$ be such that for any $x \neq x'$, $d(x, x') > C$. Let $d'_f$ be defined as in Corollary \ref{metricquotientls} and let $d_f$ be the quotient metric on $Y$ with respect to $f$. Suppose $d_f(y, y') < R$. Then there is a sequence of pairs of points $(a_i, b_i)_{1 \leq i \leq k}$ such that $f(a_1) = y,\ f(b_n) = y'$ and $f(b_i) = f(a_{i+1})$, and such that
$$
 \sum_{i=1}^{k} d_X(a_i, b_i) \leq R.
$$
Since $X$ is uniformly discrete, this means that $kC \leq R$, which implies that
$$
k+\sum_{i=1}^{k} d_X(a_i, b_i) \leq R+R/C
$$
from which it follows that $d'_f(y,y') \leq R + R/C$. On the other hand, if $d'_f(y,y') < R$ then it is easy to check that $d_f(y,y') < R$, so we have that $d_f$ and $d'_f$ induce the same large scale structure.
\end{proof}

Finally, we have the following characterization of weak coarse quotient maps between metric spaces.

\begin{Proposition}\label{charweak}
Let $f: X \rightarrow Y$ be a large scale continuous map between non-empty metric spaces.  Then the following are equivalent:
\begin{itemize}
\item[(a)]$f$ is a weak coarse quotient map;
\item[(b)] there exists a $T > 0$ such that for every $R > 0$ there is an $S(R) > 0$ and an integer $n(R)$ such that if $d_Y(y, y') \leq R$ for $y,y' \in Y$ then there is a sequence of pairs of points $(a_i, b_i)_{1 \leq i \leq n(R)}$ in $X$ such that $d_Y(f(a_1),y) \leq T,\ d_Y(f(b_n), y') \leq T$, and $d_Y(f(b_i), f(a_{i+1}))\leq T$ and $d_X(a_i, b_i) \leq S(R)$ for all $i$.
\end{itemize}
\end{Proposition}
\begin{proof}
(a) $\Rightarrow$ (b): Suppose $f$ is a weak coarse quotient map with quotient scale $\mathcal{V}$. We claim that $T = \mathsf{mesh}(\st(\mathcal{V}, \mathcal{V}))$ works. Let $d'_Y$ be the metric on $Y$ constructed in Proposition \ref{metricweakquotient} which we know induces the large scale structure on $Y$. Thus for every $R > 0$ there is an $R' > 0$ such that $d_Y(y, y') \leq R \implies d'_Y(y, y') \leq R'$ for all $y,y\in Y$ where $d_Y$ is the original metric on $Y$. By construction of the metric $d'_Y$, if $d'_Y(y,y') \leq R'$ then there must be a sequence of pairs of points $(a_i, b_i)_{1 \leq i \leq n}$ in $X$ with $n \leq R+1$ such that $f(a_1)\mathcal{V}y,\ f(b_n)\mathcal{V}y'$, $f(b_i) \st(\mathcal{V}, \mathcal{V}) f(a_{i+1})$ for all $i$ and $d(a_i, b_i) \leq R+1$ for all $i$. Thus setting $n(R) >  R+1$, $S(R) = R+1$ we have the result.

(b) $\Rightarrow$ (a): Clearly $f$ must be coarsely surjective with $Y \subseteq B(f(X), T)$. Let $R > 0$ and pick $S(R)$ and $n(R)$ as in (b). Let $\mathcal{B}_{S(R)}$ and $\mathcal{B}_{T}$ be covers of $X$ and $Y$ by $S(R)$-balls and $T$-balls respectively. If $d(y,y') \leq R$ for $y,y' \in Y$ then (b) implies that $y$ and $y'$ are connected by a chain of at most $n(R)$ elements of $\st(f(\mathcal{B}_{S(R)}), \mathcal{B}_T)$. Since $n(R)$ and $S(R)$ depend only on $R$, the cover of $Y$ by $R$-balls is an element of the large scale structure generated by $f(\mathcal{X})$ and $\mathcal{B}_T$ where $\mathcal{X}$ is the large scale structure on $X$. It follows that $f$ is a weak coarse quotient map with quotient scale $\mathcal{B}_T$.
\end{proof}

Note that we can obviously choose $n(R) = S(R)$ for every $R$ in (b) of Proposition \ref{charweak}, but it is more intuitively clear to keep the two quantities separate.

\end{document}